\documentclass{article}

\usepackage[a4paper, margin=3cm]{geometry}
\usepackage{amsmath,amssymb,amsthm}
\usepackage{mathtools}
\usepackage[backend=biber, style=numeric]{biblatex}
\usepackage{hyperref}
\usepackage{graphicx}
\usepackage{tikz-cd}
\usepackage{tikz}
\usepackage{mathrsfs}
\usepackage{xurl}

\theoremstyle{definition} 

\newcommand{\ch}{\operatorname{ch}}

\newcommand{\tr}{\operatorname{Tr}}

\newcommand{\Gr}{\operatorname{Gr}}
\newcommand{\End}{\operatorname{End}}
\newcommand{\Res}{\operatorname{Res}}
\newcommand{\Hom}{\operatorname{Hom}}

\usepackage{enumitem}
\usepackage{mathrsfs}

\theoremstyle{definition}
\newtheorem{theorem}{Theorem}[section]
\newtheorem{proposition}[theorem]{Proposition}
\newtheorem{lemma}[theorem]{Lemma}

\newtheorem{remark}[theorem]{Remark}
\newtheoremstyle{italicdefinition}
  {3pt}
  {3pt}
  {\normalfont\itshape}
  {}
  {\bfseries}
  {.}
  {.5em}
  {}
\theoremstyle{italicdefinition}
\newtheorem{definition}[theorem]{Definition}
\newtheorem{convention}[theorem]{Convention}
\theoremstyle{definition}

\providecommand{\End}{\operatorname{End}}
\renewcommand{\Hom}{\operatorname{Hom}}
\renewcommand{\tr}{\operatorname{tr}}
\providecommand{\rk}{\operatorname{rk}}
\providecommand{\im}{\operatorname{Im}}

\providecommand{\Res}{\operatorname{Res}}
\providecommand{\Ad}{\operatorname{Ad}}
\providecommand{\ad}{\operatorname{ad}}

\renewcommand{\Gr}{\operatorname{Gr}}

\providecommand{\Aut}{\operatorname{Aut}}

\providecommand{\Id}{\operatorname{Id}}

\providecommand{\Ric}{\operatorname{Ric}}
\providecommand{\PU}{\mathrm{PU}}

\providecommand{\PGL}{\mathrm{PGL}}
\providecommand{\GL}{\mathrm{GL}}
\providecommand{\Lat}{\mathsf{Lat}}
\providecommand{\BG}{\mathsf{BG}}

\newcommand{\C}{\mathbb C}
\newcommand{\R}{\mathbb R}

\newcommand{\Z}{\mathbb Z}
\newcommand{\OX}{\mathcal O_X}
\newcommand{\cX}{\mathcal X}
\newcommand{\cD}{\mathcal D}
\newcommand{\cE}{\mathcal E}
\newcommand{\bB}{\mathbb B}

\renewcommand{\ch}{\operatorname{ch}}
\newcommand{\parch}{\operatorname{par\!-ch}}
\newcommand{\parc}{\operatorname{par\!-c}}
\newcommand{\pardeg}{\operatorname{par\!-deg}}
\newcommand{\orbch}{\operatorname{orb\!-ch}}

\title{Orbifold Uniformization of Complex Algebraic Varieties via Polystable Parabolic Higgs Bundles}

\author{
  Tianshu Jiang\thanks{E-mail: jts2021@mail.ustc.edu.cn. School of Mathematical Sciences, University of Science and Technology of China, Hefei 230026, P. R. China.}
  \and
  Jiayu Li\thanks{E-mail: jiayuli@ustc.edu.cn. School of Mathematical Sciences, University of Science and Technology of China, Hefei 230026, P. R. China.}
}

\date{\today}

\begin{document}

\frenchspacing

\maketitle

\tableofcontents

\begin{abstract}
Let \(X\) be a smooth complex projective variety of dimension \(n\geq 2\), and let
\[
   D=D^p+D^c,\qquad D^c=\sum_{i\in I}D_i^c,
\]
be a simple normal crossing divisor. We regard \(D^p\) as the cusp divisor and the components \(D_i^c\) as compact orbifold divisors with standard weights
\[
   \alpha_i=1-\frac1{p_i}=\frac{p_i-1}{p_i}.
\]
Let
\[
   \cX=X\bigl[\sqrt[p_i]{D_i^c}\bigr]_{i\in I}
\]
be the root stack along the compact components. We study the canonical parabolic Higgs bundle
\[
   E_*=\bigl(\Omega_X^1(\log D^p)\oplus\mathcal O_X\bigr)_*,
\]
whose only non-trivial compact weights are the weights \(\alpha_i\) on the conormal lines of the \(D_i^c\), while the parabolic structure along \(D^p\) is trivial. Equivalently, this is the canonical log-root Higgs bundle on the log-root pair \((\cX,D^p)\), where \(D^p\) is kept as an ordinary logarithmic boundary.

Assume that \((E_*,\theta)\) is polystable with respect to some ample line bundle and that equality holds in the parabolic Bogomolov--Gieseker inequality. We prove that the trace-free adjoint Higgs bundle is flat. The associated principal \(\PU(n,1)\)-variation gives a faithful monodromy representation
\[
   \rho:\pi_1^{\rm orb}(\cX^o)\longrightarrow \PU(n,1)
\]
and a period map to the complex ball \(\bB^n\). For these standard weights, this period map is unramified in the orbifold sense and identifies the log-root pair with the canonical orbifold toroidal compactification of a finite-volume complex ball quotient:
\[
   (\cX,D^p)\simeq (\mathcal T_\Gamma,D_\Gamma^{\rm tor})
\]
in the sense of Definition~\ref{def:canonical-orbifold-toroidal}, where
\[
   \Gamma=\rho\bigl(\pi_1^{\rm orb}(\cX^o)\bigr)\subset \PU(n,1)
\]
is a finite-volume lattice satisfying the regular log-root condition of Definition~\ref{def:regular-lattice-category}.

We formulate this standard-weight uniformization as an equivalence of categories. Let \(\Lat_n^{\rm reg}\) be the regular log-root lattice category of Definition~\ref{def:regular-lattice-category}, and let \(\BG_n^{\rm std}\) be the standard Bogomolov--Gieseker category of Definition~\ref{def:BGcategory}. The compactified quotient construction and the monodromy construction from the canonical parabolic Higgs bundle define quasi-inverse functors
\[
   \Lat_n^{\rm reg}\simeq \BG_n^{\rm std}.
\]
\end{abstract}

\noindent\textbf{2020 Mathematics Subject Classification.}
Primary 32Q30; Secondary 14A20, 14D07, 14J60, 32G20, 32M15, 53C07.

\medskip
\noindent\textbf{Keywords.}
Parabolic Higgs bundles, root stacks, complex hyperbolic geometry, ball quotients, Bogomolov--Gieseker equality.

\section{Introduction}

\subsection{Equality and converse statements}

Simpson's compact ball-uniformization theorem \cite{Simpson1988} is the starting point. In the compact case one considers the system of Hodge bundles
\[
   \bigl(\Omega_X^1\oplus\mathcal O_X,\theta\bigr),\qquad
   \theta(a,b)=(0,a),
\]
and the equality case of the Bogomolov--Gieseker inequality forces the corresponding trace-free adjoint Higgs bundle to be flat. The resulting principal \(\PU(n,1)\)-variation has a period map to \(\mathbb B^n\). We prove a log-root version of this picture, with logarithmic cusp boundary, compact orbifold divisors, and rational parabolic weights. The standard-weight case will be singled out after the parabolic Higgs bundle has been defined.

Let \(X\) be a smooth complex projective variety and let
\[
   D=D^p+D^c,\qquad D^c=\sum_{i\in I}D_i^c,
\]
be a simple normal crossing divisor. The divisor \(D^p\) is a cusp divisor and is removed from the open part. The compact divisor \(D^c\) is not removed; instead its components carry finite root-stack stabilizers. We set
\[
   \cX=X\bigl[\sqrt[p_i]{D_i^c}\bigr]_{i\in I},\qquad
   \pi:\cX\to X,
\]
and
\[
   \cX^o:=\cX\setminus \pi^{-1}(D^p).
\]
The cusp divisor \(D^p\) is treated as an ordinary logarithmic boundary. Thus the geometric object is the log-root pair \((\cX,D^p)\).

Choose rational compact weights
\[
   \alpha_i=\frac{q_i}{p_i}\in[0,1),\qquad (p_i,q_i)=1.
\]
We consider
\[
   E=\Omega_X^1(\log D^p)\oplus\mathcal O_X,
\]
with the tautological Higgs field
\[
   \theta(a,b)=(0,a),
\]
taking values in \(\Omega_X^1(\log(D^p+D^c))\). The associated parabolic Higgs bundle is denoted by \((E_*,\theta)\). Its only non-trivial compact weights are the weights \(\alpha_i\) on the conormal lines of the divisors \(D_i^c\), and the parabolic structure along \(D^p\) is trivial. The standard weights are the special case
\[
   \alpha_i=1-\frac1{p_i},
\]
equivalently \(q_i=p_i-1\).

The proof of the equality direction proceeds in three steps. First, Mochizuki's correspondence for parabolic Higgs bundles \cite[Theorem~1.4]{Mochizuki2006} gives an adapted harmonic metric on the trace-free adjoint Higgs bundle, and the parabolic Bogomolov--Gieseker equality forces the Hitchin--Simpson curvature of this adjoint bundle to vanish. Second, a principal reconstruction step upgrades the flat adjoint system to a flat principal \(\PU(n,1)\)-variation. The resulting period map is controlled by the Kodaira--Spencer morphism; in the standard-weight case this morphism is an isomorphism on the root stack, so the period map is unramified in the orbifold sense. Finally, cusp metric estimates based on Mochizuki's asymptotic analysis \cite{Mochizuki2002}, together with the metric rigidity theorem of Deng--Cadorel \cite[Theorem~A.7]{DengCadorel2022}, identify the compactification with the canonical orbifold toroidal compactification of a ball quotient.

The main standard-weight statement is the following.

\begin{theorem}\label{thm:standard-uniformization}
Let \((\cX,D^p)\) be as above, with standard weights \(\alpha_i=1-1/p_i\) along the compact components \(D_i^c\). Assume that the canonical parabolic Higgs bundle \((E_*,\theta)\) is polystable with respect to some ample line bundle \(L\), and that equality holds in the parabolic Bogomolov--Gieseker inequality
\[
   \left(\frac{\parch_1(E_*)^2}{2(n+1)}-\parch_2(E_*)\right)c_1(L)^{n-2}=0.
\]
Then the trace-free adjoint Higgs bundle is flat and gives a faithful representation
\[
   \rho:\pi_1^{\rm orb}(\cX^o)\longrightarrow \PU(n,1).
\]
The associated period map identifies the orbifold universal cover of \(\cX^o\) with \(\bB^n\). Hence
\[
   \Gamma:=\rho\bigl(\pi_1^{\rm orb}(\cX^o)\bigr)\subset \PU(n,1)
\]
is a finite-volume lattice, and
\[
   (\cX,D^p)\simeq (\mathcal T_\Gamma,D_\Gamma^{\rm tor})
\]
as canonical orbifold toroidal compactifications in the sense of Definition~\ref{def:canonical-orbifold-toroidal}. In particular, \(D^p\) is smooth and each connected component of \(D^p\) is a finite quotient of an abelian variety.
\end{theorem}

The proof actually gives a more general branched statement for arbitrary compact weights. We record it separately, because Theorem~\ref{thm:standard-uniformization} is the unramified case which enters the categorical equivalence.

\begin{proposition}\label{prop:branched-equality}
Let \(D_i^c\) carry a rational weight \(q_i/p_i\), with \(0\le q_i<p_i\) and \((p_i,q_i)=1\). Assume that \((E_*,\theta)\) is \(L\)-polystable and that equality holds in the parabolic Bogomolov--Gieseker inequality
\begin{equation}\label{eq:BG}
\left(\frac{\parch_1(E_*)^2}{2(n+1)}-\parch_2(E_*)\right)c_1(L)^{n-2}=0.
\end{equation}
Then the following statements hold.

\begin{enumerate}[label=\textup{(\arabic*)}]
\item The trace-free endomorphism Higgs bundle
\[
\bigl(\End_0(\cE),\theta_{\End_0}\bigr)
\]
is flat. It carries an adapted tame harmonic metric and is induced by a flat principal \(\PU(n,1)\)-bundle. Thus one obtains a representation
\[
\rho:\pi_1^{\rm orb}(\cX^o)\longrightarrow \PU(n,1)
\]
and a \(\rho\)-equivariant surjective holomorphic period map
\[
\mathcal{P}:\widetilde{\cX^o}\longrightarrow \PU(n,1)/U(n)\simeq \bB^n.
\]
Moreover, the orbifold universal cover \(\widetilde{\cX^o}\) is represented by an ordinary simply connected complex manifold; equivalently, all stabilizer groups on the covering orbifold are trivial.

\item The differential of \(\mathcal{P}\) is the Kodaira--Spencer morphism
\[
\tau:T_{\cX}(-\log D^p)\longrightarrow \Hom(\cE^{1,0},\cE^{0,1}).
\]
It is generically an isomorphism. Near the generic point of \(\cD_i^c\), choose the root coordinate
\[
z_i=w_i^{p_i}.
\]
Then the normal component of \(\mathcal{P}\) has local form
\[
\mathcal{P}_i(w)=u_i(w)w_i^{p_i-q_i},\qquad u_i(0)\ne 0.
\]
Thus, on the smooth root chart of the universal cover, \(\mathcal{P}\) is ramified along the lift of \(\cD_i^c\) with ramification index \(p_i-q_i\). Equivalently, if \(\widetilde{\cD_i^c}\) denotes the inverse image of \(\cD_i^c\), then the ramification divisor is
\[
R_{\mathcal{P}}=\sum_{i\in I}(p_i-q_i-1)\,\widetilde{\cD_i^c}.
\]
If the equivariant period map descends to a quotient by a discrete holonomy group, the induced coarse map has the same branch index along \(D_i^c\).

\item The period map is a local biholomorphism in the orbifold sense if and only if \(q_i=p_i-1\) for every \(i\). In this special case it is an unramified orbifold ball uniformization. In general it is a branched period map, not an unramified orbifold uniformization.

\item The cusp divisor \(D^p\) is forced to be smooth: its irreducible components are pairwise disjoint.

\item Around \(D^p\), the local monodromy is parabolic in \(\PU(n,1)\). Around \(D_i^c\), the local orbifold stabilizer has order \(p_i\), and its image is elliptic of order \(p_i\) in the normal period coordinate. More generally, at a point lying on the compact components \(\cD_i^c\), \(i\in S\), the local stabilizer \(\prod_{i\in S}\mu_{p_i}\) maps injectively to \(\PU(n,1)\); on the corresponding normal period lines its action is given by
\[
   \xi_i\longmapsto \zeta_i^{p_i-q_i}\xi_i.
\]
\end{enumerate}
\end{proposition}

We also use the converse direction for branched complex-hyperbolic structures. It shows, in particular, that the canonical log-root toroidal compactification attached to a regular lattice satisfies the polystability and Bogomolov--Gieseker equality conditions.

For the converse, put
\[
   X^\circ:=X\setminus(D^p\cup D^c),\qquad
   \Delta:=D^p+\sum_{i\in I}\frac{q_i}{p_i}D_i^c.
\]
A polarization class
\(\alpha\in H^{1,1}(X,\mathbb R)\) is called \(\Delta\)-admissible if it is big and nef and contains a closed positive \((1,1)\)-current \(T\in\alpha\) such that \(T|_{X^\circ}\) is a smooth K\"ahler form and has at most mixed Poincar\'e--cone growth of type \(\Delta\) in the sense of currents. Definition~\ref{def:admissible} gives the precise formulation used in Section~\ref{sec:converse}.

A branched complex-hyperbolic structure of type \(\Delta\) on the weighted pair \((X,D^p,D^c,\{q_i/p_i\}_{i\in I})\) consists of the orbifold universal cover \(\widetilde{\cX^o}\to\cX^o\), represented by a simply connected smooth complex manifold, a representation \(\rho:\pi_1^{\rm orb}(\cX^o)\to \PU(n,1)\), and a \(\rho\)-equivariant period map \(\mathcal{P}:\widetilde{\cX^o}\to\bB^n\). The differential is non-degenerate away from the inverse image of \(D^c\), and near a lift of \(D_i^c\), in the root coordinate \(x_i=w_i^{p_i}\) and a normal ball coordinate \(\xi_i\), one has
\[
   \xi_i\circ\mathcal{P}=u_i(w)w_i^{p_i-q_i},
   \qquad u_i(0)\ne0;
\]
while along \(D^p\) the descended hyperbolic metric has at most Poincar\'e cusp growth.

\begin{proposition}\label{prop:converse}
Assume that \((X,D^p,D^c,\{q_i/p_i\})\) carries a branched complex-hyperbolic structure of type \(\Delta\). Let \(\omega_{\mathbb B}\) be the invariant Bergman metric on \(\mathbb B^n\), normalized by
\[
   \Ric(\omega_{\mathbb B})=-(n+1)\omega_{\mathbb B},
\]
and let \(\omega_{\rm br}\) be its descended pull-back. Then:
\begin{enumerate}[label=\textup{(\arabic*)}]
\item \(\omega_{\rm br}\) is a smooth K\"ahler metric on \(X^\circ\) and has at most mixed Poincar\'e--cone growth of type \(\Delta\).
\item The current
\[
   \Theta:=\frac{n+1}{2\pi}\,\omega_{\rm br}
\]
extends across \(|D|\) as a closed positive current representing
\[
   c_1\left(K_X+D^p+\sum_{i\in I}\frac{q_i}{p_i}D_i^c\right)=c_1(K_X+\Delta).
\]
Moreover, the class \(c_1(K_X+\Delta)\) is big and nef. In particular it is \(\Delta\)-admissible.
\item For every \(\Delta\)-admissible big and nef class \(\alpha\), the parabolic Higgs bundle \((E_*,\theta)\) is parabolic \(\mu_\alpha\)-polystable. In particular, it is parabolic \(\mu_{c_1(K_X+\Delta)}\)-polystable.
\item The parabolic Chern equality holds as a cohomology class:
\[
   \frac{\parch_1(E_*)^2}{2(n+1)}-\parch_2(E_*)=0.
\]
Equivalently,
\[
   2\parc_2(E_*^{1,0})-\frac{n}{n+1}\parc_1(E_*^{1,0})^2=0.
\]
\end{enumerate}
\end{proposition}

\subsection{Standard log-root lattices and categories}

We now turn to the standard-weight quotient side. The first point is to make precise which compactification of a non-neat finite-volume ball quotient is used. The cusp boundary is treated logarithmically, while the compact elliptic divisors are retained as root-stack divisors.

\begin{definition}\label{def:canonical-orbifold-toroidal}
Let \(\Gamma\subset\PU(n,1)\) be a finite-covolume lattice and put
\[
   U_\Gamma:=\bB^n/\Gamma.
\]
Choose a finite-index neat normal subgroup \(\Gamma'\triangleleft\Gamma\), put \(F=\Gamma/\Gamma'\), and let
\[
   U':=\bB^n/\Gamma'.
\]
Let \((\overline U_{\Gamma'}^{\rm tor},D_{\Gamma'}^{\rm tor})\) be the smooth toroidal compactification of the neat quotient. The action of \(F\) on \(U'\) extends to \(\overline U_{\Gamma'}^{\rm tor}\). The coarse canonical toroidal pair of \(U_\Gamma\) is
\[
   (\overline U_\Gamma^{\rm tor},D_\Gamma^{\rm tor})
   :=
   (\overline U_{\Gamma'}^{\rm tor}/F,D_{\Gamma'}^{\rm tor}/F),
\]
which is independent of the choice of \(\Gamma'\) after passing to a common neat normal subgroup.

Assume that the quotient has regular log-root local form. Namely, near the toroidal boundary and the compact elliptic mirrors the quotient is locally described by coordinates
\[
   q=t^m,\qquad z_i=w_i^{p_i},\qquad y=y.
\]
Here \(q=0\) is the ordinary logarithmic cusp boundary, while \(z_i=0\) are the compact elliptic divisors with stabilizers \(\mu_{p_i}\). The cusp Kummer directions and the compact root directions are required not to mix.

Let \(D_i^c\subset \overline U_\Gamma^{\rm tor}\) be the closures of the compact elliptic divisorial locus, with root orders \(p_i\). The canonical orbifold toroidal compactification of \([\bB^n/\Gamma]\), relative to the parabolic boundary, is the Deligne--Mumford stack
\[
   \mathcal T_\Gamma
   :=
   \overline U_\Gamma^{\rm tor}\Bigl[\sqrt[p_i]{D_i^c}\Bigr]_{i\in I},
\]
together with the ordinary logarithmic boundary divisor \(D_\Gamma^{\rm tor}\). Possible finite Kummer phenomena along the cusp boundary are absorbed by the logarithmic structure and are not retained as root-stack inertia along \(D_\Gamma^{\rm tor}\).
\end{definition}

Proposition~\ref{prop:converse} gives the direction from a standard ball quotient to the parabolic Bogomolov--Gieseker side.

\begin{theorem}\label{thm:standard-orbifold-converse}
Let \(q_i=p_i-1\) for every \(i\), and put
\[
   \Delta=D^p+\sum_{i\in I}\left(1-\frac1{p_i}\right)D_i^c.
\]
Assume that \((\cX,D^p)\) is the canonical orbifold toroidal compactification of \([\bB^n/\Gamma]\) in the sense of Definition~\ref{def:canonical-orbifold-toroidal}, with local stabilizer order \(p_i\) along \(\cD_i^c\). Then \((X,D^p,D^c,\{(p_i-1)/p_i\})\) carries the standard unramified orbifold complex-hyperbolic structure of type \(\Delta\). Hence \(K_X+\Delta\) is big and nef, its first Chern class is \(\Delta\)-admissible, and for every \(\Delta\)-admissible big and nef class \(\alpha\), the parabolic Higgs bundle \((E_*,\theta)\) is parabolic \(\mu_\alpha\)-polystable and satisfies the parabolic Chern equality.
\end{theorem}

We next make precise the two categories which appear in the standard-weight statement. The lattice side is restricted to the finite-volume lattices for which the toroidal compactification has the log-root normal-crossing local form required by the consequence of Proposition~\ref{prop:branched-equality}.

\begin{definition}[Regular log-root lattices]\label{def:regular-lattice-category}
Let \(\Lat_n^{\rm reg}\) be the following category. An object is a finite-covolume lattice \(\Gamma\subset\PU(n,1)\), taken up to \(\PU(n,1)\)-conjugacy, with the following regular log-root property. Choose a finite-index neat normal subgroup \(\Gamma'\triangleleft\Gamma\), put \(F=\Gamma/\Gamma'\), and let \((\overline U_{\Gamma'}^{\rm tor},D_{\Gamma'}^{\rm tor})\) be the smooth toroidal compactification of \(\bB^n/\Gamma'\). The finite group \(F\) acts on \(\overline U_{\Gamma'}^{\rm tor}\). We require that, near every point over the toroidal boundary and the compact elliptic mirrors, this action have product log-root normal-crossing local form
\[
   q=t^m,\qquad z_i=w_i^{p_i},\qquad y=y,
\]
where \(q=0\) is the ordinary logarithmic cusp boundary and the divisors \(z_i=0\) are compact root divisors with stabilizers \(\mu_{p_i}\). In particular, the cusp Kummer directions and the compact root directions do not mix, and the only compact stabilizers are products of the cyclic reflection groups \(\mu_{p_i}\).

A morphism \(\Gamma_1\to\Gamma_2\) in \(\Lat_n^{\rm reg}\) is a finite-index inclusion
\[
   \Gamma_1\subset g\Gamma_2g^{-1}
\]
for some \(g\in\PU(n,1)\), taken up to simultaneous \(\PU(n,1)\)-conjugacy, and required to preserve the above log-root structures on the associated toroidal compactifications.
\end{definition}

\begin{definition}[The standard Bogomolov--Gieseker category]\label{def:BGcategory}
Let \(\BG_n^{\rm std}\) be the following category. An object is a smooth log-root pair
\[
   \cX^{\log}:=\bigl(X[\sqrt[p_i]{D_i^c}]_{i\in I},D^p\bigr)
\]
such that \(X\) is a smooth projective variety of dimension \(n\), the divisor \(D^p+\sum_iD_i^c\) is simple normal crossing, \(D^p\) is an ordinary logarithmic boundary, and the compact components \(D_i^c\) carry the standard weights
\[
   \alpha_i=1-\frac1{p_i}.
\]
Let
\[
   E_*=(\Omega_X^1(\log D^p)\oplus\mathcal O_X)_*
\]
be the canonical parabolic Higgs bundle whose only non-trivial compact weights are the weights \(\alpha_i\) on the conormal lines of the \(D_i^c\), and whose parabolic structure along \(D^p\) is trivial. We require that \((E_*,\theta)\) be parabolic polystable with respect to some ample polarization \(L\), and that it satisfy the parabolic Bogomolov--Gieseker equality
\[
   \left(\frac{\parch_1(E_*)^2}{2(n+1)}-\parch_2(E_*)\right)c_1(L)^{n-2}=0.
\]
Smoothness of \(D^p\) is not included in the definition, but it follows from the equality condition by Proposition~\ref{prop:branched-equality}.

In this case, we can also regard \((E_*,\theta)\) as the canonical log-root Higgs bundle
\[
   \mathbb E_{\cX}:=
   \bigl(\Omega^1_{\cX}(\log D^p)\oplus\mathcal O_{\cX},\theta_{\cX}\bigr),
   \qquad
   \theta_{\cX}(a,b)=(0,a).
\]
A morphism \(f:\cX_1^{\log}\to\cX_2^{\log}\) in \(\BG_n^{\rm std}\) is a finite surjective representable morphism of log-root stacks which satisfies
\[
   \bigl(f^{-1}(|D_2^p|)\bigr)_{\rm red}=|D_1^p|,
\]
and is a local isomorphism in the log-root sense, meaning that the natural logarithmic differential
\[
   df_{\log}:
   f^*\Omega^1_{\cX_2}(\log D_2^p)
   \longrightarrow
   \Omega^1_{\cX_1}(\log D_1^p)
\]
is an isomorphism.
\end{definition}

Theorem~\ref{thm:standard-uniformization}, Theorem~\ref{thm:standard-orbifold-converse}, and the covering properties of these morphisms give the categorical statement below.

\begin{theorem}\label{thm:categorical}
The compactified quotient construction defines an equivalence of categories
\[
   Q_{\rm tor}:\Lat_n^{\rm reg}\longrightarrow \BG_n^{\rm std}.
\]
A quasi-inverse is given by the monodromy construction
\[
   U_{\rm BG}(\cX^{\log})=
   \rho_{\cX}\bigl(\pi_1^{\rm orb}(\cX^o)\bigr)\subset\PU(n,1),
\]
where \(\rho_{\cX}\) is the representation obtained from the canonical parabolic Higgs bundle in the equality case. Under this equivalence, finite-index inclusions of regular log-root lattices correspond to finite representable morphisms of log-root pairs which are local isomorphisms in the log-root sense.
\end{theorem}

\begin{remark}\label{rem:deng-cadorel-comparison}
The case \(D^c=\emptyset\) with \(D^p\) smooth is treated in \cite{DengCadorel2022}. Here we assume only that \(D^p+D^c\) has simple normal crossings. The smoothness of \(D^p\) is forced by the rigidity of the vanishing of the Bogomolov--Gieseker discriminant; see Proposition~\ref{prop:Dp-smooth}.

A related one-dimensional reconstruction theorem was recently obtained by Lin and Sheng \cite{LinSheng2026}. In dimension one the negativity of the log-orbifold Euler characteristic already implies the relevant stability condition, and no Bogomolov--Gieseker equality is needed. Moreover, every finite-covolume lattice in \(\PU(1,1)\) has the regular log-root form used here. In higher dimensions the boundary and elliptic strata of finite-volume ball quotients are much more complicated, and the regular lattice category considered in this paper is only a first clean case. It is natural to ask what intrinsic geometric conditions recover the lattice in greater generality.
\end{remark}

\section{Parabolic and Root-Stack Set-Up}\label{sec:setup-root}

\subsection{Basic set-up and notation}

Let \(X\) be a smooth complex projective variety of dimension \(n\). Let
\[
D=D^p+D^c
\]
be a simple normal crossing divisor, where \(D^c=\sum_{i\in I}D_i^c\).
The divisor \(D^p\) is the cusp divisor and is removed from the open part. The divisor \(D^c\) is the compact orbifold divisor and carries the finite stabilizer data.

Put
\[
E=\Omega_X^1(\log D^p)\oplus \OX.
\]
The first summand is denoted by \(E^{1,0}\), and the second by \(E^{0,1}\). We consider the logarithmic Higgs field
\[
\theta:E\longrightarrow E\otimes \Omega_X^1(\log D),\qquad (a,b)\longmapsto (0,a).
\]
Equivalently, \(\theta\) is the tautological map from \(E^{1,0}\) to \(E^{0,1}\otimes \Omega_X^1(\log D)\).

For each compact component \(D_i^c\), choose a rational number
\[
\alpha_i=\frac{q_i}{p_i}\in [0,1),\qquad \gcd(p_i,q_i)=1.
\]
The parabolic structure along \(D_i^c\) has a single non-trivial step: the conormal line
\[
N^*_{D_i^c/X}\subset \Omega_X^1|_{D_i^c}\subset E|_{D_i^c}
\]
has weight \(q_i/p_i\), while all other directions have weight \(0\). Along \(D^p\) the parabolic structure is trivial. The resulting parabolic Higgs bundle is denoted by \((E_*,\theta)\).

Let
\[
\cX=X\bigl[\sqrt[p_i]{D_i^c}\bigr]_{i\in I}
\]
be the iterated root stack associated with the compact divisors \(D_i^c\) and orders \(p_i\). Let
\[
\pi:\cX\longrightarrow X
\]
be the coarse moduli morphism, and let \(\cD_i^c\subset \cX\) be the reduced stacky divisor lying over \(D_i^c\). We write
\[
\cX^o:=\cX\setminus \pi^{-1}(D^p),\qquad X^o:=X\setminus D^p.
\]
The parabolic Higgs bundle \((E_*,\theta)\) corresponds, by the root-stack/parabolic correspondence \cite{Iy-Si}, to an orbifold Higgs bundle
\[
(\cE,\theta_{\cE})=(\cE^{1,0}\oplus \cE^{0,1},\theta_{\cE})
\]
on \(\cX\), with \(\cE^{0,1}=\mathcal O_{\cX}\). The local form of \(\cE^{1,0}\) is described in Section~\ref{sec:local-root}. In general \(\cE^{1,0}\) is not equal to \(\Omega^1_{\cX}(\log D^p)\); equality holds along \(D_i^c\) precisely when \(q_i=p_i-1\).

\subsection{Root-stack local conventions}\label{sec:local-root}

Fix the following local convention. Let \(x\in X\), and suppose that, near \(x\), the compact components of \(D^c\) passing through \(x\) are
\[
D_{i_1}^c,\dots,D_{i_s}^c,\qquad i_\nu\in I.
\]
Choose local coordinates \((z_1,\dots,z_n)\) such that
\[
D_{i_\nu}^c=\{z_\nu=0\}\quad(1\le \nu\le s).
\]
The root stack \(\cX\) has a local chart with coordinates \((w_1,
\dots,w_s,z_{s+1},\dots,z_n)\), where
\[
z_\nu=w_\nu^{p_{i_\nu}}\quad(1\le \nu\le s),
\]
and with stabilizer group \(\prod_{\nu=1}^s\mu_{p_{i_\nu}}\) acting by
\[
\zeta_\nu\cdot w_\nu=\zeta_\nu w_\nu.
\]

\begin{convention}\label{conv:root-frame}
Along \(D_i^c\), the parabolic weight \(q_i/p_i\) on the conormal line is represented on the root chart by the equivariant frame
\begin{equation}\label{eq:eta-root}
\eta_i:=w_i^{-q_i}dz_i.
\end{equation}
Since \(z_i=w_i^{p_i}\), this is
\begin{equation}\label{eq:eta-dw}
\eta_i=p_i w_i^{p_i-q_i-1}dw_i.
\end{equation}
\end{convention}

The frame \(\eta_i\) is a regular frame of the orbifold vector bundle \(\cE^{1,0}\) associated with the parabolic bundle. Formula~\eqref{eq:eta-dw} describes the natural morphism from this orbifold bundle to the ordinary cotangent bundle of the root chart.

\begin{proposition}\label{prop:E10-vs-cotangent}
There is a natural morphism
\[
\iota:\cE^{1,0}\longrightarrow \Omega^1_{\cX}(\log D^p)
\]
which is generically an isomorphism. Along \(\cD_i^c\), its determinant vanishes to order \(p_i-q_i-1\). In particular, \(\cE^{1,0}=\Omega^1_{\cX}(\log D^p)\) near \(\cD_i^c\) if and only if \(q_i=p_i-1\).
\end{proposition}

\begin{proof}
The assertion is local on the root chart. In the normal direction to \(D_i^c\), the bundle \(\cE^{1,0}\) has frame \(\eta_i\), while \(\Omega^1_{\cX}\) has frame \(dw_i\). By \eqref{eq:eta-dw}, the morphism \(\iota\) sends
\[
\eta_i\longmapsto p_i w_i^{p_i-q_i-1}dw_i.
\]
Thus its vanishing order in this normal direction is \(p_i-q_i-1\). The remaining non-orbifold directions have weight zero, and along \(D^p\) the frame is the usual logarithmic frame. Therefore the determinant of \(\iota\) vanishes along \(\cD_i^c\) with the asserted order. It is nowhere vanishing along \(\cD_i^c\) precisely when \(p_i-q_i-1=0\), equivalently \(q_i=p_i-1\).
\end{proof}

The root-stack Higgs field is obtained by composing the tautological section with \(\iota\):
\[
\theta_{\cE}:\cE^{1,0}\longrightarrow \cE^{0,1}\otimes \Omega^1_{\cX}(\log D^p).
\]
In the normal direction to \(D_i^c\), it sends the frame \(\eta_i\) to
\[
1\otimes p_i w_i^{p_i-q_i-1}dw_i.
\]
Thus the Higgs field is regular along \(D^c\), but it degenerates there unless \(q_i=p_i-1\).

\subsection{The orbifold Higgs field and induced metrics}

Around a point where the compact components are \(D_{i_1}^c,\ldots,D_{i_r}^c\) and the cusp components are \(P_1,\ldots,P_s\), choose coordinates such that
\[
D_{i_\nu}^c=\{z_\nu=0\}\quad(1\le \nu\le r),
\qquad
P_\mu=\{z_{r+\mu}=0\}\quad(1\le \mu\le s).
\]
Use the root chart \(z_\nu=w_\nu^{p_{i_\nu}}\) for \(1\le \nu\le r\). The \((1,0)\)-part of the orbifold bundle has local frames
\[
   \eta_\nu=w_\nu^{-q_{i_\nu}}dz_\nu
   =p_{i_\nu}\,w_\nu^{p_{i_\nu}-q_{i_\nu}-1}dw_\nu\quad(1\le \nu\le r),
\]
\[
   \eta_j=d\log z_j\quad(r+1\le j\le r+s),
   \qquad
   \eta_a=dz_a\quad(a>r+s),
\]
and \(\cE^{0,1}\) has frame \(e_0=1\). The orbifold Higgs field is
\[
   \theta_{\cE}(\eta_b)=e_0\otimes\iota(\eta_b),
   \qquad
   \theta_{\cE}(e_0)=0,
\]
where \(\iota:\cE^{1,0}\to\Omega^1_{\cX}(\log D^p)\) is the morphism of Proposition~\ref{prop:E10-vs-cotangent}. In particular, in a compact root direction it sends
\[
   \eta_\nu\longmapsto e_0\otimes
   p_{i_\nu}\,w_\nu^{p_{i_\nu}-q_{i_\nu}-1}dw_\nu.
\]
This formula is equivariant for the stabilizer action and therefore defines a logarithmic orbifold Higgs field on \((\cX,D^p)\). It is regular along the compact stacky divisor and has logarithmic poles only along \(D^p\).

These frames also describe adapted metrics. If \(h\) is an adapted harmonic metric for the associated filtered Higgs bundle on the coarse complement, then its pull-back to the root chart, written in the frames \(\eta_i,d\log z_j,dz_a,e_0\), is invariant under the finite stabilizer group and has the prescribed orbifold growth. Conversely, an orbifold adapted harmonic metric written in these frames gives the tame adapted metric on \(X\setminus(D^p\cup D^c)\). The precise smooth extension across the compact root divisors is recorded in Proposition~\ref{prop:root-stack-KH}. On the complement the change of frame is obtained by multiplying by the prescribed root powers, so the Hitchin--Simpson connection and its curvature transform by gauge conjugation.

\section{The Equality Case and the Period Map}\label{sec:equality-period}

\subsection{Mochizuki's correspondence in root-stack form}\label{sec:root-kh}

Mochizuki's theorem is stated for regular filtered Higgs bundles on a smooth complex manifold with a simple normal crossing divisor. For the root-stack formulation we use the following reduction.

\begin{proposition}\label{prop:root-stack-KH}
Let \((\mathcal F,\varphi)\) be a logarithmic orbifold Higgs bundle on \((\cX,D^p)\), and let \((F_*,\varphi)\) be the associated locally abelian regular filtered Higgs bundle on the smooth pair \((X,D^p+D^c)\). If \((F_*,\varphi)\) is \(\mu_L\)-polystable with trivial parabolic characteristic numbers, then \((\mathcal F,\varphi)|_{\cX^o}\) carries an adapted tame harmonic metric. In local root frames this metric is the pull-back of Mochizuki's adapted harmonic metric for \((F_*,\varphi)\), and it extends smoothly across the compact root divisors on every finite root chart. The only remaining tame logarithmic boundary is the inverse image of \(D^p\). The usual uniqueness statement for adapted harmonic metrics also holds in this root-stack form: on stable summands the metric is unique up to a positive constant, and on polystable isotypical summands the ambiguity is a constant Hermitian metric on the multiplicity space.
\end{proposition}

\begin{proof}
The root-stack/parabolic correspondence identifies vector bundles on \(\cX\) with locally abelian parabolic bundles on \((X,D^c)\) with the prescribed denominators; the corresponding Chern characters and degrees agree \cite{Iy-Si}. Adding \(D^p\) only records the logarithmic boundary with trivial parabolic filtration. Hence \((\mathcal F,\varphi)\) is equivalently a regular filtered Higgs bundle on \((X,D^p+D^c)\).

Mochizuki's Kobayashi--Hitchin correspondence for tame harmonic bundles \cite[Theorem~1.4]{Mochizuki2006} gives an adapted pluri-harmonic metric on \(F|_{X\setminus(D^p\cup D^c)}\), unique up to the usual ambiguity. In a root chart, if \(e_a\) has weight \(m_{a,i}/p_i\) along \(D_i^c\), the corresponding orbifold frame is
\[
   \widetilde e_a=\prod_i w_i^{-m_{a,i}}\pi^*e_a.
\]
In these frames the singular powers prescribed by the parabolic weights are absorbed by the root coordinates. Thus the pulled-back metric and its inverse are locally bounded across the compact root divisors \(w_i=0\). The pulled-back Higgs field is smooth in the compact root directions. Logarithmic poles remain only along the inverse image of \(D^p\).

Next we prove that the pullback harmonic metric extends smoothly across the compact root divisors. Let \(U\) be a finite root chart with coordinates \((w,z,y)\), let
\[
   B=\{w_1\cdots w_r=0\}
\]
be the compact root divisor, and let \(U^\dagger=U\setminus D^p\). It is enough to prove smoothness on every relatively compact open set \(U_0\Subset U^\dagger\); thus the logarithmic boundary \(D^p\) plays no role in the following local argument. Choose a holomorphic frame of \(\mathcal F\) on \(U_0\). The metric is represented on \(U_0^*:=U_0\setminus B\) by a positive Hermitian matrix \(H\), and the adaptedness just recalled gives
\[
   C^{-1}I\le H\le CI.
\]
Write the Higgs field as \(\varphi=\sum_j A_j d\xi_j\) in the root coordinates \(\xi_j\). Since \(\varphi\) has no pole along \(B\), all matrices \(A_j\) are smooth and bounded on \(U_0\).

We now explain the removable regularity statement used below. In the frame chosen above the pluri-harmonic equation is the Hitchin--Simpson flatness equation. For local regularity it is enough to use its contraction with a fixed smooth background Hermitian metric on the coordinate ball; this contracted $(1,1)$-equation can be written locally as
\begin{equation}\label{eq:local-harmonic-metric-equation}
   \sum_\mu \partial_{\bar\xi_\mu}\bigl(H^{-1}\partial_{\xi_\mu}H\bigr)
   +
   \sum_\mu \bigl[A_\mu,H^{-1}A_\mu^{*}H\bigr]=0,
\end{equation}
where \(A_\mu^{*}\) denotes the adjoint with respect to the fixed background Hermitian metric of the chosen frame. This formula is to be understood on \(U_0^*\). After choosing normal coordinates on the symmetric space
\[
   \mathcal P_N:=\GL(N,\mathbb C)/U(N)
\]
centered at a point in the compact ball containing the image of \(H\), it has the form
\begin{equation}\label{eq:forced-harmonic-map-system}
   L H^a+\Gamma^a_{bc}(H)\,\langle dH^b,dH^c\rangle_g=F^a(\xi,H).
\end{equation}
Here \(L\) is the scalar Laplace operator of a fixed smooth background Hermitian metric \(g\) on the coordinate ball, \(\Gamma^a_{bc}\) are the Christoffel symbols of \(\mathcal P_N\), and the term \(F\) is obtained from the Higgs commutator in \eqref{eq:local-harmonic-metric-equation}. Since the root powers have been absorbed into the orbifold frame, the matrices \(A_\mu\) are smooth up to \(B\). Since \(H\) and \(H^{-1}\) are bounded, \(F\) is smooth in \(\xi\), smooth in \(H\), and uniformly bounded on \(U_0\times K\), where \(K\Subset\mathcal P_N\) is the compact geodesic ball containing \(H(U_0^*)\). When \(F=0\), \eqref{eq:forced-harmonic-map-system} is exactly the weak harmonic-map system considered by Hildebrandt--Kaul--Widman \cite[Theorem~3]{HildebrandtKaulWidman1977}. In our case the principal part is the same; the Higgs field contributes only the bounded zeroth-order forcing term \(F\).

We shall use the following local removable regularity fact, which is a minor forced variant of the proof of \cite[Theorem~3]{HildebrandtKaulWidman1977}.

\smallskip
\noindent\emph{Local removable regularity assertion.}
Let \(V\Subset U_0\) be a coordinate ball and let \(B\cap V\) be a normal crossing divisor. Let
\[
   u:V\setminus B\longrightarrow \mathcal P_N
\]
be a \(C^2\) solution of \eqref{eq:forced-harmonic-map-system} whose image lies in a compact regular geodesic ball \(K\Subset\mathcal P_N\). Assume that \(F\) is bounded on \(V\times K\). If \(u\) is bounded in \(K\), then \(u\) extends to a smooth solution on \(V\).

Indeed, first one obtains the Caccioppoli estimate. Let \(Q\in K\) be the center of a regular geodesic ball containing the image, let \(\mathcal r_Q^2\) be the squared distance to \(Q\), and test the weak form of \eqref{eq:forced-harmonic-map-system} on \(V\setminus B\) by \(\chi^2\nabla\mathcal r_Q^2(u)\). The Hessian comparison estimate used in \cite[Section~4]{HildebrandtKaulWidman1977} is valid here because \(\mathcal P_N\) is simply connected with non-positive sectional curvature, hence geodesic balls are regular. It gives the usual positive term controlling \(\chi^2|du|^2\). The forcing term contributes
\[
   \left|\int_{V\setminus B}\chi^2\langle F(\xi,u),\nabla\mathcal r_Q^2(u)\rangle\right|
   \le C\int_{V\setminus B}\chi^2,
\]
because \(F\) and \(\nabla\mathcal r_Q^2\) are bounded on \(V\times K\). The cross term is absorbed by Cauchy's inequality. Therefore
\begin{equation}\label{eq:forced-caccioppoli}
   \int_{V\setminus B}\chi^2|du|^2
   \le C\int_{V\setminus B}|d\chi|^2+C\int_{V\setminus B}\chi^2
\end{equation}
for all \(\chi\in C_c^\infty(V\setminus B)\).

Choose logarithmic cut-off functions \(\rho_\varepsilon\) for the divisor \(B\cap V\). For one component \(\{w=0\}\), \(\rho_\varepsilon=0\) on \(|w|\le\varepsilon^2\), \(\rho_\varepsilon=1\) on \(|w|\ge\varepsilon\), and \(\rho_\varepsilon\) is linear in \(\log |w|\) on the annulus \(\varepsilon^2<|w|<\varepsilon\); for a normal crossing divisor take the product over the components. Then
\[
   0\le\rho_\varepsilon\le1,
   \qquad \rho_\varepsilon\to1 \text{ on } V\setminus B,
   \qquad \int_V |d\rho_\varepsilon|^2\to0 .
\]
Putting \(\chi=\psi\rho_\varepsilon\) in \eqref{eq:forced-caccioppoli} and letting \(\varepsilon\to0\) gives \(u\in W^{1,2}_{\rm loc}(V)\). Testing the equation against \(\psi\rho_\varepsilon\Phi\), with \(\Phi\) an arbitrary compactly supported vector-valued test function in the normal coordinates of \(K\), and then letting \(\varepsilon\to0\), shows that \(u\) satisfies \eqref{eq:forced-harmonic-map-system} weakly on all of \(V\). The terms involving \(d\rho_\varepsilon\) vanish by Cauchy--Schwarz, the bound just obtained for \(du\), and \(\int |d\rho_\varepsilon|^2\to0\).

It remains to prove continuity of this weak solution. The proof of \cite[Theorem~3]{HildebrandtKaulWidman1977} applies with only harmless changes. More explicitly, in the Green-function estimates of \cite[Section~4]{HildebrandtKaulWidman1977}, the weak harmonic-map identity is tested against the Green kernel multiplied by the normal-coordinate vector field. The new forcing term is bounded by
\[
   C\int_{B_R(x_0)}G(x,x_0)\,dx,
\]
which is \(O(R^2)\) for real dimension at least three and \(O(R^2|\log R|)\) in real dimension two; in either case it tends to zero as \(R\to0\). The corresponding annular estimate used in their proof acquires the same vanishing error. Thus their oscillation estimate still gives
\[
   \lim_{R\to0}\operatorname*{ess\,osc}_{B_R(x_0)}u=0
\]
for every \(x_0\in V\). Hence \(u\) is continuous. Once continuity is known, the equation is a uniformly elliptic system with natural quadratic growth \(|du|^2\) and an additional bounded right-hand side. The regularity step invoked after \cite[Theorem~3]{HildebrandtKaulWidman1977}, namely the local quasilinear elliptic regularity for such systems, gives H\"older continuity of \(u\); then the standard elliptic bootstrap in the smooth target coordinates of \(K\) gives \(u\in C^\infty(V)\). This proves the local assertion.

We now apply the assertion to the harmonic metric \(H\). The two-sided estimate \(C^{-1}I\le H\le CI\) says exactly that the map \(H:U_0^*\to\mathcal P_N\) has image in a compact geodesic ball. The symmetric space \(\mathcal P_N\) has non-positive sectional curvature, so this ball is regular. The preceding local assertion therefore shows that \(H\) extends smoothly across \(B\) on every relatively compact coordinate ball \(V\Subset U_0\). Since \(V\) was arbitrary, the pulled-back metric extends smoothly across the compact root divisor on the root chart. The extension is unique, hence is invariant under the finite stabilizer group, and therefore defines a smooth orbifold harmonic metric on \(\cX^o\). The local gauge change by the root powers conjugates the Hitchin--Simpson connection on the complement, so harmonicity and flatness are unchanged by passing between the parabolic and root-stack descriptions.

Finally, uniqueness is inherited from Mochizuki's uniqueness statement. Indeed, two orbifold adapted harmonic metrics give, in root frames, two invariant adapted metrics for the same regular filtered Higgs bundle on the coarse complement. Mochizuki's uniqueness identifies them on stable summands up to positive constants, and in the polystable case up to constant Hermitian metrics on multiplicity spaces. Pulling this statement back to the root charts gives exactly the asserted orbifold uniqueness.
\end{proof}

\subsection{Flatness of the adjoint Higgs bundle}

Let
\[
r:=\rk(E)=n+1.
\]
The parabolic Chern characters of \(E_*\) agree with the orbifold Chern characters of \(\cE\) on \(\cX\). We use the notation \(\parch\) on \(X\) and \(\orbch\) on \(\cX\).

\begin{lemma}\label{lem:end-ch}
For an orbifold vector bundle \(F\) of rank \(r\),
\[
\orbch_1(\End F)=0,
\]
and
\begin{equation}\label{eq:ch2-end}
\orbch_2(\End F)=2r\,\orbch_2(F)-\orbch_1(F)^2.
\end{equation}
Consequently, equality \eqref{eq:BG} is equivalent to
\begin{equation}\label{eq:ch2-end-vanish}
\orbch_2(\End \cE)c_1(L)^{n-2}=0.
\end{equation}
The same equality holds for \(\End_0(\cE)\).
\end{lemma}

\begin{proof}
The identity follows from
\[
\ch(\End F)=\ch(F)\ch(F^\vee).
\]
Writing
\[
\ch(F)=r+c_1(F)+\ch_2(F)+\cdots,
\qquad
\ch(F^\vee)=r-c_1(F)+\ch_2(F)+\cdots,
\]
the degree-one term cancels, and the degree-two term is
\[
r\ch_2(F)+r\ch_2(F)-c_1(F)^2.
\]
This gives \eqref{eq:ch2-end}. The equality
\[
\left(\frac{c_1(F)^2}{2r}-\ch_2(F)\right)c_1(L)^{n-2}=0
\]
is equivalent to
\[
\bigl(2r\ch_2(F)-c_1(F)^2\bigr)c_1(L)^{n-2}=0.
\]
Taking \(F=\cE\) gives \eqref{eq:ch2-end-vanish}. Finally,
\[
\End(\cE)=\End_0(\cE)\oplus\mathcal O_{\cX},
\]
and the trivial summand contributes no first or second Chern character.
\end{proof}

\begin{proposition}\label{prop:flat-adjoint}
Under the assumptions of Proposition~\ref{prop:branched-equality}, the logarithmic orbifold Higgs bundle
\[
\bigl(\End_0(\cE),\theta_{\End_0}\bigr)
\]
on \((\cX,D^p)\) is polystable with vanishing characteristic numbers. Hence it admits an adapted tame harmonic metric. Moreover, the parabolic Chern--Weil identity forces the Hitchin--Simpson connection of any adapted Hermitian--Einstein metric on this adjoint Higgs bundle, with the usual determinant normalization on stable summands, to be flat.
\end{proposition}

\begin{proof}
Polystability is preserved by duals, tensor products and direct summands for locally abelian parabolic Higgs bundles. Hence \(\End(\cE)=\cE\otimes\cE^\vee\) and its direct summand \(\End_0(\cE)\) are polystable.

By Lemma~\ref{lem:end-ch},
\[
\orbch_1(\End_0(\cE))=0,
\qquad
\orbch_2(\End_0(\cE))c_1(L)^{n-2}=0.
\]
Therefore the associated regular filtered Higgs bundle has trivial parabolic characteristic numbers. Proposition~\ref{prop:root-stack-KH} gives the existence of an adapted pluri-harmonic metric on \(\End_0(\cE)|_{\cX^o}\).

Let \(h\) be any adapted Hermitian--Einstein metric on this adjoint Higgs bundle, with the determinant normalized on the stable summands, and let
\[
\nabla_h=\partial_h+\bar\partial+\theta_{\End_0}+\theta_{\End_0,h}^{\dagger}
\]
be its Hitchin--Simpson connection. The parabolic Chern--Weil identity expresses the Bogomolov--Gieseker defect as a positive constant times the \(L^2\)-norm of the trace-free Hitchin--Simpson curvature. Since the defect vanishes and the bundle is trace-free adjoint, there is no residual central curvature term. Hence \(\nabla_h\) is flat. The resulting local system is semisimple by the harmonic-metric correspondence.
\end{proof}

\subsection{The principal \texorpdfstring{$\PU(n,1)$}{PU(n,1)} system and the period map}

Let \(V=V^{1,0}\oplus V^{0,1}\) be a complex vector space with
\[
\dim V^{1,0}=n,
\qquad
\dim V^{0,1}=1.
\]
Put
\[
G=\PGL(V),
\qquad
K=\mathrm P\bigl(\GL(V^{1,0})\times \GL(V^{0,1})\bigr).
\]
Let \(G_0=\PU(n,1)\) be the real form preserving a Hermitian form of signature \((n,1)\), and let
\[
K_0=\mathrm P(U(n)\times U(1))\simeq U(n)
\]
be its maximal compact subgroup. The Lie algebra \(\mathfrak g=\mathfrak{pgl}(V)\) has the Hodge decomposition
\[
\mathfrak g=\mathfrak g^{-1,1}\oplus \mathfrak g^{0,0}\oplus \mathfrak g^{1,-1},
\]
where
\[
\mathfrak g^{-1,1}=\Hom(V^{1,0},V^{0,1}),
\qquad
\mathfrak g^{1,-1}=\Hom(V^{0,1},V^{1,0}).
\]

Let \(P\) be the orbifold principal \(K\)-bundle of frames of
\[
   \cE=\cE^{1,0}\oplus\cE^{0,1}
\]
which preserve the Hodge decomposition, modulo the common scalar. Thus
\[
   P\times_K\mathfrak g^{-1,1}\simeq
   \Hom(\cE^{1,0},\cE^{0,1}).
\]
The Higgs field \(\theta_{\cE}\) has type \((-1,1)\), and therefore defines the principal Higgs field
\[
   \tau\in H^0\bigl(\cX,\Omega^1_{\cX}(\log D^p)\otimes P\times_K\mathfrak g^{-1,1}\bigr),
\]
or equivalently the Kodaira--Spencer morphism
\[
   \tau:T_{\cX}(-\log D^p)\longrightarrow P\times_K\mathfrak g^{-1,1}.
\]
On the associated adjoint bundle \(P\times_K\mathfrak g\), the vector-bundle Higgs field induced by the principal Higgs field \(\tau\) is
\[
   \ad(\tau):P\times_K\mathfrak g
   \longrightarrow (P\times_K\mathfrak g)\otimes\Omega^1_{\cX}(\log D^p),
   \qquad a\longmapsto [\tau,a].
\]
Thus \(\tau\) is the Higgs field of the principal system, while \(\ad(\tau)\) denotes only its induced adjoint Higgs field.

\begin{lemma}\label{lem:finite-adjoint-reconstruction}
Let \(V=V^{1,0}\oplus V^{0,1}\), with \(\dim V^{1,0}=n\) and \(\dim V^{0,1}=1\). Put
\[
   G=\PGL(V),\qquad
   K=\mathrm P\bigl(\GL(V^{1,0})\times\GL(V^{0,1})\bigr),
   \qquad \mathfrak g=\mathfrak{pgl}(V),
\]
and let
\[
   \mathfrak g=\mathfrak g^{-1,1}\oplus\mathfrak g^{0,0}\oplus\mathfrak g^{1,-1}
\]
be the Hodge decomposition induced by \(V^{1,0}\oplus V^{0,1}\). Fix a Hermitian form \(S_0\) of signature \((n,1)\), positive on \(V^{1,0}\) and negative on \(V^{0,1}\), and let \(h^0_{\mathfrak g}\) be the induced positive Hodge metric on \(\mathfrak g\). Then the stabilizer of \(h^0_{\mathfrak g}\) in \(K\) is
\[
   K_0=\mathrm P\bigl(U(V^{1,0},S_0)\times U(V^{0,1},-S_0)\bigr).
\]
Moreover, if \(H\) is a positive Hermitian metric on \(\mathfrak g\) satisfying the following three conditions:
\begin{enumerate}[label=\textup{(\roman*)}]
\item the three Hodge summands are mutually orthogonal for \(H\);
\item \(H\) is compatible with the Killing form, so that \(\mathfrak g^{-1,1}\) and \(\mathfrak g^{1,-1}\) are dual through it;
\item the metric on \(\mathfrak g^{0,0}\) is the one determined, via the bracket action, by the induced actions on \(\mathfrak g^{-1,1}\oplus\mathfrak g^{1,-1}\),
\end{enumerate}
then \(H\) lies in the \(K\)-orbit of \(h^0_{\mathfrak g}\). Consequently the space of such metrics is the homogeneous space
\[
   K\cdot h^0_{\mathfrak g}\simeq K/K_0.
\]
\end{lemma}

\begin{proof}
Write \(W=V^{1,0}\) and \(L=V^{0,1}\). Since \(\dim L=1\), a point of \(K\) is represented by a pair \((A,\lambda)\in \GL(W)\times\GL(L)\), modulo simultaneous scalar multiplication. A pair of positive Hermitian metrics \((h_W,h_L)\) on \((W,L)\), again modulo simultaneous scalar multiplication, determines a positive Hodge metric on \(\mathfrak g\) as follows. On
\[
   \mathfrak g^{-1,1}=\Hom(W,L),\qquad
   \mathfrak g^{1,-1}=\Hom(L,W),
\]
it is the usual Hom metric induced by \(h_W\) and \(h_L\), and the two summands are dual through the Killing form. On
\[
   \mathfrak g^{0,0}=\bigl(\End(W)\oplus\End(L)\bigr)/\C\cdot \Id_V
\]
it is the metric induced by the Hermitian metric on \(\End(W)\) defined by \(h_W\); the \(\End(L)\)-factor is scalar and disappears after quotienting by the common scalar. This construction is exactly the \(K\)-orbit of \(h^0_{\mathfrak g}\), and the stabilizer of the model metric is \(K_0\). Hence the orbit is \(K/K_0\).

It remains to see that the conditions in the statement force a metric to be of this form. The restriction of \(H\) to \(\mathfrak g^{-1,1}=W^\vee\otimes L\) is a positive Hermitian metric on \(W^\vee\otimes L\). Choose a positive metric \(h_L\) on the one-dimensional space \(L\). Then the formula
\[
   h_W^\vee(\alpha,\beta)
   :=\frac{H_{-1,1}(\alpha\otimes\ell,\beta\otimes\ell)}{h_L(\ell,\ell)}
\]
for any non-zero \(\ell\in L\) defines a positive Hermitian metric \(h_W^\vee\) on \(W^\vee\), hence a metric \(h_W\) on \(W\). Changing \(h_L\) by a positive scalar changes \(h_W\) by the inverse scalar, so the pair \((h_W,h_L)\) is well-defined modulo simultaneous scalar multiplication. Condition \textup{(ii)} forces the metric on \(\mathfrak g^{1,-1}\) to be the Hom metric induced by the same pair. Condition \textup{(iii)} then forces the metric on \(\mathfrak g^{0,0}\) to be the one induced by \(h_W\), because the bracket identifies \(\mathfrak g^{0,0}\) with the trace-free endomorphisms acting on \(W\). Hence \(H\) is the metric associated with \((h_W,h_L)\), and therefore belongs to \(K\cdot h^0_{\mathfrak g}\).
\end{proof}

\begin{lemma}\label{lem:adjoint-hodge-tensor}
For the above \(P\) and \(\tau\), put
\[
   \overline{\mathcal A}:=P\times_K\mathfrak g,
   \qquad
   \mathcal A:=\overline{\mathcal A}|_{\cX^o},
   \qquad
   \theta_{\mathcal A}:=\ad(\tau)|_{\cX^o}.
\]
Let \(h\) be an adapted pluri-harmonic metric on \((\mathcal A,\theta_{\mathcal A})\) whose Hitchin--Simpson connection is flat. After replacing \(h\) within the usual uniqueness ambiguity on the polystable isotypical summands, \(h\) is induced by a smooth adapted \(K_0\)-reduction
\[
   P_H\subset P|_{\cX^o},
   \qquad K_0=\mathrm P(U(n)\times U(1)).
\]
For this reduction, the adjoint connection of the principal \(G_0=\PU(n,1)\)-connection
\[
   \nabla^{P_H}=\nabla^c_{P_H}+\tau+\tau^{\dagger}_{P_H}
\]
is the Hitchin--Simpson connection of \((\mathcal A,\theta_{\mathcal A},h)\).
\end{lemma}

\begin{proof}
The metric \(h\) is smooth on finite root charts away from the inverse image of \(D^p\), by Proposition~\ref{prop:root-stack-KH}, and remains tame along \(D^p\). We write its flat Hitchin--Simpson connection as
\[
   \nabla_h=\partial_h+\bar\partial+\theta_{\mathcal A}+\theta_{\mathcal A,h}^{\dagger}.
\]

We first record the functoriality statement used below. Let \((F,\theta_F,h_F)\) and \((G,\theta_G,h_G)\) be two adapted pluri-harmonic bundles on \(\cX^o\) arising from polystable logarithmic orbifold Higgs bundles on \((\cX,D^p)\) with trivial characteristic numbers. Let
\[
   \phi:F\longrightarrow G
\]
be a holomorphic Higgs morphism which extends across \(D^p\) as a morphism of the corresponding logarithmic orbifold Higgs bundles. Then \(\phi\) is parallel for the induced flat connection on \(\Hom(F,G)\). Indeed, on
\[
   H:=\Hom(F,G),
   \qquad
   \theta_H(u)=\theta_G\circ u-(u\otimes 1)\circ\theta_F,
\]
the induced metric \(h_H\) is pluri-harmonic, and the section \(\phi\) satisfies
\[
   \bar\partial_H\phi=0,
   \qquad
   \theta_H(\phi)=0.
\]
The Bochner identity for a harmonic bundle gives, on \(\cX^o\),
\[
   \sqrt{-1}\Lambda\partial\bar\partial |\phi|_{h_H}^2
   = |(\partial_{h_H}+\theta_{H,h_H}^{\dagger})\phi|_{h_H}^2.
\]
Choose a complete Poincar\'e type K\"ahler metric on \(\cX^o\) and the standard cut-off functions near \(D^p\). Since \(\phi\) extends as a logarithmic orbifold morphism and the metrics are adapted, \(|\phi|_{h_H}\) is bounded; for the tensorial morphisms used below this boundedness is immediate from the fact that the tensors have weight zero in adapted frames. Multiplying the preceding identity by the square of a cut-off function \(\chi_\varepsilon\) and integrating by parts gives
\[
   \int_{\cX^o}\chi_\varepsilon^2
   |(\partial_{h_H}+\theta_{H,h_H}^{\dagger})\phi|_{h_H}^2
   \leq
   C\int_{\operatorname{supp} d\chi_\varepsilon}|d\chi_\varepsilon|^2 |\phi|_{h_H}^2.
\]
The right hand side tends to zero by the usual Poincar\'e cut-off estimate. Hence
\[
   (\partial_{h_H}+\theta_{H,h_H}^{\dagger})\phi=0.
\]
Together with \(\bar\partial_H\phi=0\) and \(\theta_H(\phi)=0\), this says precisely that \(\phi\) is flat. The same proof applies to morphisms between any tensor constructions of such bundles, equipped with the tensor product and dual metrics.

We first show that \(h\) may be chosen to be a Hodge metric. The adjoint Higgs bundle is a system of Hodge bundles
\[
   \overline{\mathcal A}
   =\overline{\mathcal A}^{-1,1}\oplus
     \overline{\mathcal A}^{0,0}\oplus
     \overline{\mathcal A}^{1,-1},
   \qquad
   \theta_{\mathcal A}:\mathcal A^{p,-p}\longrightarrow
   \mathcal A^{p-1,-p+1}\otimes\Omega^1_{\cX}(\log D^p)|_{\cX^o}.
\]
For \(t\in U(1)\), let \(u_t\) act on \(\overline{\mathcal A}^{p,-p}\) by multiplication by \(t^{-p}\). Then \(u_t\) is an isomorphism from \((\overline{\mathcal A},\theta_{\mathcal A})\) to \((\overline{\mathcal A},t\theta_{\mathcal A})\). Since \(|t|=1\), a pluri-harmonic metric for \(\theta_{\mathcal A}\) is also a pluri-harmonic metric for \(t\theta_{\mathcal A}\). Thus \(u_t^*h\) is again an adapted pluri-harmonic metric for \((\mathcal A,\theta_{\mathcal A})\). By the uniqueness assertion in Proposition~\ref{prop:root-stack-KH}, on each stable summand the metric is unique up to a positive constant, and on an isotypical summand \(S\otimes M\) the ambiguity is exactly a constant Hermitian metric on the multiplicity space \(M\). Averaging these finite-dimensional multiplicity metrics over the compact group \(U(1)\), and leaving the stable metrics fixed, gives another adapted pluri-harmonic metric. Replacing \(h\) by this normalized metric, we may assume that \(u_t\) is unitary for every \(t\in U(1)\). Therefore the three Hodge summands are mutually orthogonal for \(h\).

We now impose the algebraic tensors defining the adjoint group. We first fix the notation. If
\(W\) is a finite-dimensional vector space, by a tensor construction in \(W\) we mean a finite direct sum of spaces of the form
\[
   W^{\otimes a}\otimes (W^\vee)^{\otimes b}.
\]
Thus the Lie bracket of \(\mathfrak g\) is a point of
\[
   \Hom(\wedge^2\mathfrak g,\mathfrak g)
   \subset \mathfrak g\otimes(\mathfrak g^\vee)^{\otimes 2},
\]
and the Killing form is a point of \(\operatorname{Sym}^2\mathfrak g^\vee\). We choose finitely many tensor constructions \(T_\alpha(\mathfrak g)\) and tensors
\[
   \mathbf t_\alpha\in T_\alpha(\mathfrak g)
\]
with the following property:
\[
   \Ad(G)=\{A\in\GL(\mathfrak g)\mid A\cdot \mathbf t_\alpha=\mathbf t_\alpha
   \text{ for every }\alpha\},
   \qquad G=\PGL(V).
\]
The list is chosen to contain the Lie bracket and the Killing form. The existence of such a finite list follows from Chevalley's theorem in its standard Tannakian form: for an algebraic subgroup \(H\subset\GL(W)\), there is a tensor construction \(T(W)\) and a line \(\ell\subset T(W)\) such that \(H\) is the stabilizer of \(\ell\). Applied to \(H=\Ad(G)\subset\GL(\mathfrak g)\), and replacing the line by a finite set of defining tensors, this gives the displayed description of \(\Ad(G)\); see, for example, \cite[Proposition~2.20]{DeligneMilne1982}.

Since the tensors \(\mathbf t_\alpha\) are \(G\)-invariant, they define global tensor sections
\[
   \mathbf t_{\alpha,\mathcal A}\in H^0\bigl(\cX^o,T_\alpha(\mathcal A)\bigr).
\]
These sections are holomorphic Higgs tensors in the following sense. Locally the Higgs field on \(T_\alpha(\mathcal A)\) is induced by the infinitesimal action of \(\tau\in\mathfrak g^{-1,1}\) on the tensor construction \(T_\alpha(\mathfrak g)\). Since \(\mathbf t_\alpha\) is fixed by \(G\), every element of \(\mathfrak g\) acts trivially on it. Hence
\[
   \theta_{T_\alpha(\mathcal A)}(\mathbf t_{\alpha,\mathcal A})=0.
\]
Equivalently, \(1\mapsto \mathbf t_{\alpha,\mathcal A}\) is a Higgs morphism from the trivial harmonic bundle to \(T_\alpha(\mathcal A)\). The functoriality statement above therefore gives
\[
   \nabla_{T_\alpha(h)}\mathbf t_{\alpha,\mathcal A}=0.
\]
Thus the parallel transport of \(\nabla_h\) preserves all defining tensors. In model frames it lies in the common stabilizer of the \(\mathbf t_\alpha\), namely in \(\Ad(G)\). Consequently the flat frame bundle of \((\mathcal A,\nabla_h)\) reduces to \(\Ad(G)\). Together with the Hodge decomposition of \(\mathcal A\), this recovers the original holomorphic \(K\)-reduction \(P|_{\cX^o}\).

We now pass from the adjoint metric to a \(K_0\)-reduction of \(P|_{\cX^o}\). Choose the model Hermitian form \(S_0\) and the model adjoint Hodge metric \(h^0_{\mathfrak g}\) as in Lemma~\ref{lem:finite-adjoint-reconstruction}. That lemma identifies the relevant finite-dimensional space of adjoint Hodge metrics with
\[
   K\cdot h^0_{\mathfrak g}\simeq K/K_0.
\]

For a point \(x\in\cX^o\) and a Hodge frame \(p\in P_x\), pull back the adjoint metric \(h_x\) to the model fiber:
\[
   H_p:=p^*h_x\in \operatorname{Herm}^+(\mathfrak g).
\]
If \(p\) is replaced by \(p\cdot k\), \(k\in K\), then \(H_p\) is replaced by \(k^*H_p\). The tensors \(\mathbf t_{\alpha,\mathcal A}\) have Hodge degree zero. Since they are parallel for \(\nabla_h=(\partial_h+\bar\partial)+\theta_{\mathcal A}+\theta_{\mathcal A,h}^{\dagger}\), and since \(h\) is Hodge, the different Hodge-degree components of \(\nabla_h\mathbf t_{\alpha,\mathcal A}=0\) vanish separately. In particular, the Chern connection \(\partial_h+\bar\partial\) preserves the Hodge decomposition and the defining tensor sections. Thus, in each Hodge frame, the pulled-back metric \(H_p\) is a positive adjoint Hodge metric compatible with the defining tensors, hence belongs to the finite-dimensional orbit \(K\cdot h^0_{\mathfrak g}\simeq K/K_0\). Consequently the rule
\[
   x\longmapsto [H_p]\in K/K_0
\]
defines, independently of the local choice of \(p\), a smooth section of the associated bundle
\[
   P|_{\cX^o}\times_K (K/K_0).
\]
The standard equivalence between reductions of structure group and sections of associated homogeneous bundles gives a smooth \(K_0\)-reduction
\[
   P_H\subset P|_{\cX^o}.
\]
Equivalently, its fiber is
\[
   (P_H)_x=\bigl\{p\in P_x\mid p^*h_x=h^0_{\mathfrak g}\bigr\}.
\]
The preceding pointwise transitivity says that this set is non-empty; the ambiguity is precisely right multiplication by \(K_0\). Since \(h\) is smooth, these fibers glue to a smooth adapted principal \(K_0\)-bundle. By construction, the adjoint metric induced by \(P_H\) is exactly \(h\).

The model Hodge metric also fixes the conjugation of \(\mathfrak g\) whose fixed Lie algebra is \(\mathfrak{pu}(n,1)\). Since \(P_H\) consists of frames preserving the model metric and the Hodge decomposition, this conjugation is transported to the bundle, and \(\tau+\tau^{\dagger}_{P_H}\) is \(\mathfrak g_0\)-valued.

Let \(\nabla^c_{P_H}\) be the Chern connection of the holomorphic \(K\)-bundle \(P|_{\cX^o}\) with respect to the reduction \(P_H\). On the associated adjoint bundle, this is the Chern connection \(\partial_h+\bar\partial\). The adjoint \(\tau^{\dagger}_{P_H}\) is the \(\mathfrak g^{1,-1}\)-part determined by the same Hodge metric. Hence the adjoint connection of
\[
   \nabla^{P_H}:=\nabla^c_{P_H}+\tau+\tau^{\dagger}_{P_H}
\]
on \(P_H\times_{K_0}G_0\), where \(G_0=\PU(n,1)\), is
\[
   \partial_h+\bar\partial+\theta_{\mathcal A}+\theta_{\mathcal A,h}^{\dagger}=\nabla_h.
\]
This proves the assertion.
\end{proof}

\begin{proposition}\label{prop:principal-system}
Under the hypotheses of Proposition~\ref{prop:branched-equality}, the above principal \(K\)-Higgs system \((P,\tau)\) underlies a flat principal \(G_0=\PU(n,1)\)-variation of Hodge structure on \(\cX^o\). The associated \(K_0\)-reduction is adapted. Its Kodaira--Spencer morphism is
\[
\tau:T_{\cX}(-\log D^p)\longrightarrow P\times_K\mathfrak g^{-1,1},
\]
where \(P\) is the orbifold principal \(K\)-bundle of Hodge frames of \(\cE\). Under the identification
\[
P\times_K\mathfrak g^{-1,1}\simeq \Hom(\cE^{1,0},\cE^{0,1}),
\]
the map \(\tau\) is dual to the morphism
\[
\iota:\cE^{1,0}\longrightarrow \Omega^1_{\cX}(\log D^p)
\]
of Proposition~\ref{prop:E10-vs-cotangent}.
\end{proposition}

\begin{proof}
The principal bundle \(P\) and the principal Higgs field \(\tau\) were constructed above from the Hodge decomposition of \(\cE\). The identification
\[
P\times_K\mathfrak g^{-1,1}\simeq \Hom(\cE^{1,0},\cE^{0,1})
\]
identifies the Higgs field with contraction against the image of \(\iota\). Thus \(\tau\) is dual to \(\iota\).

Under the adjoint representation of \(G\) on \(\mathfrak g\), the principal Higgs field \(\tau\) induces the adjoint Higgs bundle
\[
\bigl(P\times_K\mathfrak g,\ad(\tau)\bigr),
\qquad
\ad(\tau)(A)=[\tau,A].
\]
The natural identification
\[
P\times_K\mathfrak g\simeq \End_0(\cE)
\]
identifies \(\ad(\tau)\) with the commutator Higgs field
\[
\theta_{\End_0}(A)=[\theta_{\cE},A].
\]
Thus \((P\times_K\mathfrak g,\ad(\tau))\) is exactly the adjoint Higgs bundle considered in Proposition~\ref{prop:flat-adjoint}.

By Proposition~\ref{prop:flat-adjoint}, the adjoint Higgs bundle just identified has an adapted pluri-harmonic metric \(h\) on \(\cX^o\) whose Hitchin--Simpson connection is flat. Applying Lemma~\ref{lem:adjoint-hodge-tensor} to this metric gives an adapted \(K_0\)-reduction \(P_H\subset P|_{\cX^o}\), and the principal connection on \(P_H\times_{K_0}\PU(n,1)\) has adjoint connection equal to the flat Hitchin--Simpson connection of \((P\times_K\mathfrak g,\ad(\tau),h)\). Since \(\PU(n,1)\) is an adjoint group, the adjoint representation is faithful; therefore the principal connection itself is flat. Hence \((P,\tau,P_H)\) is a principal \(\PU(n,1)\)-variation of Hodge structure on \(\cX^o\).
\end{proof}

\begin{proposition}\label{prop:period}
The principal variation of Proposition~\ref{prop:principal-system} gives a \(\rho\)-equivariant holomorphic period map
\[
\mathcal{P}:\widetilde{\cX^o}\longrightarrow G_0/K_0=\PU(n,1)/U(n)\simeq\bB^n.
\]
Its differential is the Kodaira--Spencer morphism \(\tau\). It is generically a local biholomorphism, and its ramification is exactly the degeneracy of \(\iota\) along \(\cD^c\).
\end{proposition}

\begin{proof}
A principal \(G_0\)-variation of Hodge structure with Hodge reduction to \(K_0\) defines a horizontal period map from the orbifold universal cover to the symmetric space \(G_0/K_0\). The differential of this map is the Higgs field, equivalently the Kodaira--Spencer morphism
\[
\tau:T_{\cX}(-\log D^p)\longrightarrow P\times_K\mathfrak g^{-1,1}.
\]
By Proposition~\ref{prop:principal-system}, \(\tau\) is dual to \(\iota\). Proposition~\ref{prop:E10-vs-cotangent} shows that \(\iota\) is generically an isomorphism and that its only divisorial degeneracy is along \(\cD^c\). Therefore \(\mathcal{P}\) is generically locally biholomorphic and is ramified precisely along the corresponding lifts of \(\cD_i^c\), with the order computed in Section~\ref{sec:branching}.
\end{proof}

\section{Local Geometry of the Period Map}\label{sec:local-geometry}

\subsection{Smoothness of the cusp divisor}

The equality case also forces the cusp divisor to be smooth. The input is Mochizuki's constancy theorem \cite[Theorem 7.1]{Mochizuki2002} for the residues of logarithmic \(\lambda\)-connections, applied at \(\lambda=0\).

\begin{theorem}\label{thm:mochizuki-cone}
Let \((F,\theta_F,h)\) be a tame nilpotent harmonic bundle with trivial parabolic structure on
\[
(\Delta^n)^*=\Delta^n\setminus \bigcup_{i=1}^{\ell}\{z_i=0\}.
\]
Let \(N_i^{(0)}=\Res_{\{z_i=0\}}(\theta_F)\) be the Higgs residue. For every
\[
a=(a_1,\dots,a_\ell)\in\R_{>0}^{\ell},
\]
put
\[
N^{(0)}(a)=\sum_{i=1}^{\ell}a_i\,N_i^{(0)}.
\]
Then the centered nilpotent weight filtration
\[
W\bigl(N^{(0)}(a)\bigr)
\]
is independent of \(a\in\R_{>0}^{\ell}\).
\end{theorem}

We shall use the following elementary weight-filtration calculation. If \(L\) is nilpotent and \(L^3=0\), then the centered monodromy weight filtration satisfies
\[
   W_{-2}(L)=\im L^2.
\]
Indeed, this is immediate from the Jordan normal form: only Jordan blocks of length three contribute to \(W_{-2}\), and on such a block the lowest weight space is exactly the image of \(L^2\).

\begin{proposition}\label{prop:Dp-smooth}
No two irreducible components of \(D^p\) meet. Equivalently, \(D^p\) is a disjoint union of smooth irreducible divisors.
\end{proposition}

\begin{proof}
Assume, for contradiction, that two components \(P_1,P_2\subset D^p\) meet transversely at a point \(x\). Choose local coordinates
\[
(z_1,z_2,z_3,\dots,z_n)
\]
centered at \(x\) such that
\[
P_1=\{z_1=0\},\qquad P_2=\{z_2=0\}.
\]
It is enough to work on the fiber over \(x\). Let
\[
e_1=\frac{dz_1}{z_1},\qquad e_2=\frac{dz_2}{z_2},\qquad e_a=dz_a\ (3\le a\le n),\qquad e_0=1
\]
be the standard local frame of \(E=\Omega_X^1(\log D^p)\oplus\OX\). Let \(N_k=\Res_{P_k}(\theta)\) for \(k=1,2\). Since \(\theta\) is the tautological map from \(E^{1,0}\) to \(E^{0,1}\otimes\Omega_X^1(\log D)\), we have
\[
N_1(e_1)=e_0,
\qquad
N_1(e_j)=0\quad(j\ne1),
\]
and
\[
N_2(e_2)=e_0,
\qquad
N_2(e_j)=0\quad(j\ne2).
\]
Thus
\[
N_1^2=N_2^2=N_1N_2=N_2N_1=0.
\]
For \(t>0\), put
\[
N_t=N_1+tN_2,
\qquad
L_t=\operatorname{ad}_{N_t}:\End_0(E_x)\longrightarrow\End_0(E_x).
\]
Then \(N_t^2=0\), and hence \(L_t^3=0\).

By Proposition~\ref{prop:flat-adjoint}, \(\End_0(\cE)\) carries an adapted tame harmonic metric. Along \(D^p\), the parabolic structure is trivial. The residues \(\operatorname{ad}_{N_1}\) and \(\operatorname{ad}_{N_2}\) are nilpotent, so the harmonic bundle is tame nilpotent with trivial parabolic structure near \(x\). Applying Theorem~\ref{thm:mochizuki-cone} gives that the weight filtration
\[
W(L_t)=W\bigl(\operatorname{ad}_{N_1+tN_2}\bigr)
\]
is independent of \(t\in\R_{>0}\).

We compute \(W_{-2}(L_t)\). For \(A\in\End_0(E_x)\),
\begin{align*}
L_t^2(A)
&=[N_t,[N_t,A]]\\
&=N_t^2A-2N_t A N_t+AN_t^2\\
&=-2N_t A N_t.
\end{align*}
Since \(N_t\) has rank one and image \(\C e_0\), the operator \(N_t A N_t\) is always a scalar multiple of \(N_t\). Conversely, define \(A_0\in\End_0(E_x)\) by
\[
A_0(e_0)=e_1,
\qquad
A_0(e_j)=0\quad(j\ne0).
\]
Then \(A_0\) has trace zero and \(N_t A_0 N_t=N_t\). Hence
\begin{equation}\label{eq:imageLt2}
\im(L_t^2)=\C\cdot N_t=\C\cdot(N_1+tN_2).
\end{equation}
By the elementary calculation above,
\[
W_{-2}(L_t)=\im(L_t^2)=\C\cdot(N_1+tN_2).
\]
The lines \(\C(N_1+N_2)\) and \(\C(N_1+2N_2)\) are distinct because \(N_1\) and \(N_2\) are linearly independent. This contradicts Mochizuki's positive-cone constancy. Therefore two components of \(D^p\) cannot meet.
\end{proof}

\subsection{Branching along the compact orbifold divisor}\label{sec:branching}

We compute the branch order on the root stack.

We shall also use the following elementary integration observation. If \(df=u(w)w^{m-1}dw\) with \(u(0)\ne0\), then, after subtracting a constant from \(f\), one has \(f(w)=v(w)w^m\) with \(v(0)\ne0\). Thus the ramification index is \(m\).

\begin{proposition}\label{prop:branching}
Fix a compact component \(D_i^c=\{z_i=0\}\), and let \(z_i=w_i^{p_i}\) be the root coordinate on \(\cX\). Then the period map has ramification index
\[
p_i-q_i
\]
along the lift of \(\cD_i^c\) to \(\widetilde{\cX^o}\). More precisely, after choosing a normal coordinate \(\mathcal{P}_i\) on \(\bB^n\),
\[
d\mathcal{P}_i=u(w)w_i^{p_i-q_i-1}dw_i,
\qquad u(0)\ne0,
\]
and hence
\[
\mathcal{P}_i(w)=v(w)w_i^{p_i-q_i},
\qquad v(0)\ne0.
\]
\end{proposition}

\begin{proof}
By Convention~\ref{conv:root-frame}, the orbifold conormal frame is
\[
\eta_i=w_i^{-q_i}dz_i.
\]
Since \(z_i=w_i^{p_i}\),
\[
\eta_i=p_i\,w_i^{p_i-q_i-1}dw_i.
\]
The differential of the period map is the Kodaira--Spencer map \(\tau\), which is dual to the morphism \(\iota:\cE^{1,0}\to\Omega^1_{\cX}(\log D^p)\). Therefore its normal component is represented, up to a non-vanishing holomorphic factor, by the one-form
\[
w_i^{p_i-q_i-1}dw_i.
\]
Integrating the normal differential gives the local normal form
\[
\mathcal{P}_i(w)=v(w)w_i^{p_i-q_i},\qquad v(0)\ne0.
\]
Hence the ramification index is \(p_i-q_i\).
\end{proof}

\begin{lemma}\label{lem:crossing-normal-form}
Let \(x\in\cX^o\) lie on the compact stacky components \(\cD_i^c\), \(i\in S\). Choose a root chart centered at \(x\), with coordinates \((w_i)_{i\in S}\) in the compact normal directions. After shrinking the chart and choosing holomorphic coordinates on \(\bB^n\) centered at \(\mathcal P(x)\), there are normal target coordinates \((\xi_i)_{i\in S}\) such that
\[
   \xi_i\circ\mathcal P_U=u_i(w)w_i^{p_i-q_i},
   \qquad u_i(0)\ne0,
\]
for every \(i\in S\).
\end{lemma}

\begin{proof}
Put \(m_i=p_i-q_i\). Let \(G_x=\prod_{i\in S}\mu_{p_i}\) be the local stabilizer. The local period map is equivariant with respect to the finite group \(\rho(\lambda_x(G_x))\), and this finite group fixes \(\mathcal P(x)\). By the holomorphic linearization theorem for finite groups acting near a fixed point, we may choose ball coordinates centered at \(\mathcal P(x)\) in which this action is linear and unitary.

In the root frame the morphism \(\iota:\cE^{1,0}\to\Omega^1_{\cX}(\log D^p)\) is diagonal in the compact normal directions:
\[
   \eta_i\longmapsto p_i\,w_i^{m_i-1}dw_i,
   \qquad i\in S,
\]
while the remaining directions are multiplied by units. Since \(d\mathcal P\) is the Kodaira--Spencer morphism dual to \(\iota\), the first non-zero term of \(d\mathcal P\) in the \(i\)-th compact normal direction is a non-zero multiple of \(w_i^{m_i-1}dw_i\). Choose a linear target coordinate \(\xi_i\) on the corresponding eigenspace for the character
\[
   \chi_i((\zeta_j)_{j\in S})=\zeta_i^{m_i}.
\]
Then \(\xi_i\circ\mathcal P\) is a \(\chi_i\)-eigenfunction for the stabilizer action. Its Taylor expansion contains only monomials whose \(G_x\)-character is \(\chi_i\). The preceding differential computation shows that the coefficient of \(w_i^{m_i}\) is non-zero. Equivariance rules out a term independent of \(w_i\), because such a term has trivial \(\mu_{p_i}\)-character. Thus
\[
   \xi_i\circ\mathcal P_U=w_i^{m_i}u_i(w)
\]
with \(u_i(0)\ne0\). This proves the simultaneous normal form for all \(i\in S\).
\end{proof}

\begin{proposition}\label{prop:cover-manifold}
Under the hypotheses of Proposition~\ref{prop:branched-equality}, the orbifold universal cover
\[
   \widetilde{\cX^o}\longrightarrow \cX^o
\]
is represented by a simply connected smooth complex manifold. Equivalently,
the covering orbifold has trivial stabilizers.
\end{proposition}

\begin{proof}
It is enough to show that every local stabilizer of \(\cX^o\) injects into
\(\pi_1^{\rm orb}(\cX^o)\). Indeed, for an orbifold covering the stabilizer
of a point in the universal cover is the kernel of the natural map from the
corresponding local stabilizer downstairs to the orbifold fundamental group.

Let \(x\in\cX^o\), and suppose that \(x\) lies on the compact stacky
components \(\cD_i^c\), \(i\in S\). In a root chart centered at \(x\), the
local stabilizer is
\[
   G_x=\prod_{i\in S}\mu_{p_i},
   \qquad
   (\zeta_i)_{i\in S}\cdot w_i=\zeta_i w_i.
\]
Let \(\lambda_x:G_x\to\pi_1^{\rm orb}(\cX^o)\) be the natural local
stabilizer homomorphism. We prove that \(\lambda_x\) is injective by showing
that \(\rho\circ\lambda_x\) is injective.

On the root chart, the local period map is equivariant:
\[
   \mathcal P_U(g\cdot w)=\rho(\lambda_x(g))\,\mathcal P_U(w),
   \qquad g\in G_x.
\]
By Lemma~\ref{lem:crossing-normal-form}, after shrinking the root chart and choosing simultaneous normal target coordinates \((\xi_i)_{i\in S}\), one has
\[
   \xi_i\circ\mathcal P_U=u_i(w)w_i^{p_i-q_i},
   \qquad u_i(0)\neq0
\]
for every \(i\in S\).
If \(g=(\zeta_i)_{i\in S}\) lies in the kernel of \(\rho\circ\lambda_x\), then
\(\mathcal P_U(g\cdot w)=\mathcal P_U(w)\). Comparing the first non-zero
normal term in the \(i\)-th coordinate gives
\[
   \zeta_i^{p_i-q_i}=1
\]
for every \(i\in S\). Since
\[
   \gcd(p_i,p_i-q_i)=\gcd(p_i,q_i)=1,
\]
we get \(\zeta_i=1\) for every \(i\in S\). Hence \(g=1\), so
\(\rho\circ\lambda_x\), and therefore \(\lambda_x\), is injective.

The only local stabilizers of \(\cX^o\) occur along the compact root divisors
\(\cD_i^c\). Thus all stabilizers inject into the orbifold fundamental group,
and the universal orbifold cover has no residual stabilizers. Hence it is
represented by a simply connected smooth complex manifold.
\end{proof}

\begin{proposition}\label{prop:coarse-local}
Suppose that the period map descends to a quotient by a discrete holonomy group \(\Gamma\subset \PU(n,1)\). Then the induced coarse map
\[
\underline{\mathcal{P}}:X^o\longrightarrow \bB^n/\Gamma
\]
has branch index \(p_i-q_i\) along \(D_i^c\).
\end{proposition}

\begin{proof}
The source coarse coordinate is \(z_i=w_i^{p_i}\). The local period coordinate on the ball is
\[
\xi_i=w_i^{p_i-q_i}.
\]
The local stabilizer \(\mu_{p_i}\) acts on \(w_i\) by multiplication and hence acts on \(\xi_i\) by the character \(\zeta\mapsto \zeta^{p_i-q_i}\). Since \(\gcd(p_i,q_i)=1\), this character has order \(p_i\). Therefore the target coarse coordinate is, up to a unit,
\[
t_i=\xi_i^{p_i}=w_i^{p_i(p_i-q_i)}=z_i^{p_i-q_i}.
\]
Thus the induced map on coarse spaces is locally
\[
z_i\longmapsto z_i^{p_i-q_i},
\]
which has branch index \(p_i-q_i\).
\end{proof}

\subsection{Metric estimates near the cusp and the compact divisor}

The equality case produces the pull-back of the invariant complex hyperbolic metric under the period map. In the standard orbifold case this is an orbifold Riemannian metric. In the general case it is a branched complex hyperbolic metric: it degenerates along the compact branch divisor \(D^c\) with the local order computed above.

\begin{lemma}\label{lem:model-metric}
Let \(x\in D^p\), and choose coordinates near \(x\) such that, by Proposition~\ref{prop:Dp-smooth},
\[
D^p=\{z_1=0\}.
\]
Put
\[
e_1=d\log z_1,
\qquad
 e_j=dz_j\quad(j\ge2),
\qquad
 e_0=1,
\]
and set \(s=-\log |z_1|^2\). Choose a local lift of the projective \(K_0\)-reduction to a Hermitian metric on \(\cE^{1,0}\oplus\cE^{0,1}\), normalized by \(|e_1|\,|e_0|=1\). In this normalization the metric is mutually bounded with
\[
|e_1|_{\tilde h}^2=s,
\qquad
|e_j|_{\tilde h}^2=1\quad(j\ge2),
\qquad
|e_0|_{\tilde h}^2=s^{-1}.
\]
Equivalently, and independently of the chosen lift, the induced metric on
\[
   P\times_K\mathfrak g^{-1,1}=\Hom(\cE^{1,0},\cE^{0,1})
\]
satisfies
\begin{equation}\label{eq:model-pullback}
\tau^*h_H\sim
\frac{\sqrt{-1}\,dz_1\wedge d\bar z_1}{|z_1|^2(\log |z_1|^2)^2}
+
\sum_{j=2}^n\frac{\sqrt{-1}\,dz_j\wedge d\bar z_j}{-\log |z_1|^2}.
\end{equation}
\end{lemma}

\begin{proof}
The principal reduction is projective, so only the metric on \(\Hom(\cE^{1,0},\cE^{0,1})\) is intrinsic. The displayed normalization is a local choice of lift; multiplying the lift by a common scalar changes the metrics on \(\cE^{1,0}\) and \(\cE^{0,1}\) simultaneously and leaves the Hom metric unchanged.

Along \(D^p\), the parabolic structure is trivial and the residues of the harmonic bundle are nilpotent. At a generic point of the cusp divisor the only non-zero residue of the Higgs field on \(E\) is
\[
   N(e_1)=e_0,
   \qquad N(e_j)=0\quad(j\ne1).
\]
The associated nilpotent-orbit model puts \(e_1\) and \(e_0\) in opposite weights, while the transverse frames \(e_j\), \(j\ge2\), have weight zero. Mochizuki's norm estimates for tame nilpotent harmonic bundles with trivial parabolic structure identify the harmonic metric, up to mutual boundedness, with this model. After the above local normalization one obtains
\[
   |e_1|^2\sim s,
   \qquad |e_0|^2\sim s^{-1},
   \qquad |e_j|^2\sim1\quad(j\ge2),
\]
where \(s=-\log|z_1|^2\).

The Kodaira--Spencer map sends
\[
z_1\frac{\partial}{\partial z_1}\longmapsto e_1^\vee\otimes e_0,
\qquad
\frac{\partial}{\partial z_j}\longmapsto e_j^\vee\otimes e_0\quad(j\ge2).
\]
The induced Hom metric is independent of the scalar normalization and satisfies
\[
|e_1^\vee\otimes e_0|^2\sim s^{-2},
\qquad
|e_j^\vee\otimes e_0|^2\sim s^{-1}\quad(j\ge2).
\]
Thus
\[
\left|\frac{\partial}{\partial z_1}\right|^2\sim \frac{1}{|z_1|^2s^2},
\qquad
\left|\frac{\partial}{\partial z_j}\right|^2\sim \frac{1}{s}\quad(j\ge2),
\]
which is \eqref{eq:model-pullback}.
\end{proof}

\begin{lemma}\label{lem:cusp-complete-volume}
The metric \(\tau^*h_H\) is complete and has finite volume near \(D^p\).
\end{lemma}

\begin{proof}
Completeness follows from the normal Poincar\'e factor. Along a path with \(|z_1|=r\to0\), the length is bounded below by a constant multiple of
\[
\int_0^\varepsilon \frac{dr}{r|\log r|}=+\infty.
\]
For finite volume, Lemma~\ref{lem:model-metric} gives a local volume density mutually bounded by
\[
\frac{dV_{\rm eucl}}{|z_1|^2|\log |z_1|^2|^{n+1}}.
\]
In polar coordinates \(z_1=re^{\sqrt{-1}\theta}\), the normal integral is
\[
\int_0^\varepsilon\frac{r\,dr}{r^2|\log r^2|^{n+1}}
=
\frac12\int_{-\log\varepsilon^2}^{+\infty}s^{-(n+1)}\,ds<\infty.
\]
\end{proof}

\begin{lemma}\label{lem:compact-branch-metric}
Near a generic point of \(\cD_i^c\), with root coordinate \(w_i\) and \(m_i=p_i-q_i\), the pull-back of the ball metric is mutually bounded in the normal direction by
\[
|w_i|^{2(m_i-1)}\,|dw_i|^2.
\]
In particular, it has finite local volume. It is non-degenerate across \(\cD_i^c\) if and only if \(m_i=1\).
\end{lemma}

\begin{proof}
By Proposition~\ref{prop:branching}, the local normal component of the period map is \(\xi_i=w_i^{m_i}\), up to multiplication by a unit and up to higher-order tangential terms. The smooth ball metric in the \(\xi_i\)-direction pulls back as
\[
|d\xi_i|^2=m_i^2|w_i|^{2(m_i-1)}|dw_i|^2,
\]
up to mutual boundedness by positive functions. The volume contribution is locally integrable because
\[
\int_0^\varepsilon r^{2(m_i-1)}r\,dr<\infty.
\]
The metric is non-degenerate at \(w_i=0\) precisely when \(m_i=1\).
\end{proof}

\subsection{The unramified standard case}\label{sec:standard-case}

In the standard orbifold case the period map is unramified, so we use the standard complete local-isometry fact: a local isometry from a complete connected Riemannian manifold to a complete connected Riemannian manifold is a covering map; if the source and target are simply connected, it is a global isometry.

\begin{proposition}\label{prop:standard-toroidal-identification}
Assume that \(q_i=p_i-1\) for every \(i\), and assume that the period map identifies
\[
   \cX^o \simeq [\bB^n/\Gamma]
\]
for a finite-volume lattice \(\Gamma\subset\PU(n,1)\). Then \((\cX,D^p)\) is the canonical orbifold toroidal compactification of \([\bB^n/\Gamma]\) in the sense of Definition~\ref{def:canonical-orbifold-toroidal}.
\end{proposition}

\begin{proof}
We first fix the three levels involved in the construction. The stack \(\cX=X[\sqrt[p_i]{D_i^c}]_{i\in I}\) has coarse moduli space \(X\), and
\[
   \cX^o=\cX\setminus\pi^{-1}(D^p),\qquad X^o=X\setminus D^p.
\]
The divisor \(D^c\) is not removed; it is encoded by stabilizers on \(\cX^o\). The auxiliary finite cover below is a cover of the quotient stack \(\cX^o\), while its compactification is taken over the coarse variety \(X\).

Choose a finite-index neat normal subgroup \(\Gamma'\triangleleft\Gamma\), put \(F=\Gamma/\Gamma'\), and set
\[
   U':=\bB^n/\Gamma'.
\]
Then \(U'\) is smooth, its parabolic elements are unipotent, and the natural map
\[
   U'=[\bB^n/\Gamma']\longrightarrow [\bB^n/\Gamma]\simeq\cX^o
\]
is finite, representable, Galois, and \'etale as a morphism of orbifolds, with group \(F\).
Let \(Y'\) be the normalization of the coarse projective variety \(X\) in the function field \(\C(U')\). Thus there is a finite morphism of normal projective varieties
\[
   \nu_X:Y'\longrightarrow X,
\]
and \(\nu_X^{-1}(X^o)=U'\). Notice that this is a morphism to the coarse space \(X\), not a morphism to the root stack \(\cX\). Over \(X^o\) the map is the coarse map induced by the representable orbifold cover \(U'\to\cX^o\). Put
\[
   E':=(\nu_X^{-1}D^p)_{\rm red}.
\]
Since \(D^c\subset X^o\), its preimage is contained in the open manifold \(U'\); it is not part of the boundary \(Y'\setminus U'\). Hence \(Y'\setminus U'=E'\).

We next verify the singularity hypothesis needed in the rigidity theorem of Deng--Cadorel~\cite[Theorem~A.7]{DengCadorel2022}. Let \(y\in Y'\), and write \(x=\nu_X(y)\). Choose analytic coordinates on \(X\) near \(x\) such that
\[
   D^p=\{z_1\cdots z_r=0\},\qquad
   D^c=\{x_1\cdots x_s=0\},
\]
with the compact component \(x_\beta=0\) carrying order \(p_\beta\). Pulling back to the root chart
\[
   x_\beta=w_\beta^{p_\beta}\qquad(1\le \beta\le s)
\]
removes the compact stabilizers. Since \(U'\to\cX^o\) is representable \'etale, its pull-back to this root chart is an ordinary finite \'etale cover away from the cusp boundary \(z_1\cdots z_r=0\). Along the cusp variables, the cover of the punctured polydisc is, after decomposing into connected components, given by finite-index subgroups of \(\Z^r\). By the analytic form of Abhyankar's lemma, after cyclic base changes
\[
   z_\alpha=t_\alpha^{m_\alpha}\qquad(1\le\alpha\le r)
\]
the normalization is finite \'etale over the smooth polydisc with coordinates \((t,w,\text{transverse})\). Therefore, locally analytically, the pair \((Y',E')\) is the quotient of a smooth simple-normal-crossing pair by a finite group preserving the boundary. In particular \((Y',E')\) has algebraic quotient singularities in the sense of \cite[Definition~A.6]{DengCadorel2022}; hence \(Y'\) has quotient singularities and is klt.

We recall the precise form of Deng--Cadorel's rigidity theorem used here. Let \(U=\bB^n/\Lambda\) with \(\Lambda\subset\PU(n,1)\) torsion-free. If \(\overline U\) is a klt compactification of \(U\), if \(B^{(1)}\) is the divisorial part of \(\overline U\setminus U\), and if the K\"ahler--Einstein metric on \(U\), regarded as a metric on \(T_{\overline U}(-\log B^{(1)})|_U\), is adapted to log order near the generic point of every component of \(B^{(1)}\), then \((\overline U,\overline U\setminus U)\) is the toroidal compactification of \(U\) \cite[Theorem~A.7]{DengCadorel2022}. In our situation \(U=U'\), \(\Lambda=\Gamma'\), \(\overline U=Y'\), and \(B^{(1)}=E'\).

It remains only to check the log-order metric hypothesis. The metric on \(U'\) obtained by pulling back the period metric is the Bergman metric of \(\bB^n/\Gamma'\). Near a generic point of a component of \(D^p\), Lemma~\ref{lem:model-metric} gives the model
\[
   \frac{\sqrt{-1}\,dz\wedge d\bar z}{|z|^2(\log |z|^2)^2}
   +\sum_{j=2}^n\frac{\sqrt{-1}\,du_j\wedge d\bar u_j}{-\log |z|^2}.
\]
After the finite normal cover \(z=t^m\), the normal Poincar\'e factor is preserved:
\[
   \frac{\sqrt{-1}\,dz\wedge d\bar z}{|z|^2(\log |z|^2)^2}
   =
   \frac{\sqrt{-1}\,dt\wedge d\bar t}{|t|^2(\log |t|^2)^2},
\]
up to the harmless constant convention in the logarithm. The tangential factors remain mutually bounded after the finite \'etale change in the transverse variables. Thus the Bergman metric is adapted to log order for \(T_{Y'}(-\log E')|_{U'}\) near the generic point of every component of \(E'\). The compact divisors \(D_i^c\) do not enter \(E'\), and on their preimages the cover is an ordinary smooth interior cover in root coordinates.

All hypotheses of Deng--Cadorel's theorem are therefore satisfied, and we obtain an isomorphism of compactifications
\[
   (Y',E')\simeq (\overline U_{\Gamma'}^{\rm tor},D_{\Gamma'}^{\rm tor}).
\]
The finite group \(F\) acts on \(U'\), and this action extends to \(Y'\) by the universal property of normalization over the fixed coarse compactification \(X\). Under the preceding identification, it is the standard extension of the \(F\)-action to the toroidal compactification, since the two extensions agree on the dense open set \(U'\) and the compactifications are normal and separated.

The finite morphism \(\nu_X:Y'\to X\) is \(F\)-invariant, hence descends to a finite morphism
\[
   Y'/F\longrightarrow X.
\]
It is an isomorphism over the dense open subset
\[
   U'/F\simeq \bB^n/\Gamma\simeq X^o.
\]
Both \(Y'/F\) and \(X\) are normal, so the finite birational morphism \(Y'/F\to X\) is an isomorphism. It carries \(E'/F\) to \(D^p\). Consequently
\[
   (X,D^p)\simeq
   (\overline U_{\Gamma'}^{\rm tor}/F,D_{\Gamma'}^{\rm tor}/F)
   =
   (\overline U_{\Gamma}^{\rm tor},D_{\Gamma}^{\rm tor})
\]
as coarse toroidal pairs.

Finally the stack structure is exactly the root structure prescribed in Definition~\ref{def:canonical-orbifold-toroidal}. The open stack is \([\bB^n/\Gamma]\), the generic stabilizer along \(D_i^c\) is \(\mu_{p_i}\), and no root structure is imposed along the cusp boundary \(D^p\). Hence
\[
   \cX=X\Bigl[\sqrt[p_i]{D_i^c}\Bigr]_{i\in I}
\]
is the canonical orbifold toroidal compactification of \([\bB^n/\Gamma]\). Since \(D_{\Gamma'}^{\rm tor}\) is a disjoint union of abelian varieties, each connected component of \(D^p\) is a finite quotient of an abelian variety.
\end{proof}

\begin{proof}[Proof of Theorem~\ref{thm:standard-uniformization}]
If \(q_i=p_i-1\) for every \(i\), Proposition~\ref{prop:E10-vs-cotangent} gives
\[
\cE^{1,0}=\Omega^1_{\cX}(\log D^p),
\]
and the Kodaira--Spencer map \(\tau\) is an isomorphism everywhere on \(\cX^o\). Therefore the period map
\[
\mathcal{P}:\widetilde{\cX^o}\to\bB^n
\]
is a local biholomorphism. Let \(\omega_H:=\tau^*h_H\) be the Hodge metric on \(\cX^o\). On the universal cover one has \(\pi^*\omega_H=\mathcal{P}^*\omega_{\bB}\), where \(\omega_{\bB}\) is the invariant metric of \(\bB^n\) with the corresponding normalization. Thus \(\mathcal P\) is a local isometry for this Hodge metric.

By Lemma~\ref{lem:cusp-complete-volume}, \(\omega_H\) is complete near \(D^p\). Away from \(D^p\), the space is covered by finitely many compact orbifold charts after shrinking near the boundary. In the standard case the compact root divisors are not branch divisors, so the metric is non-degenerate across them in the orbifold sense. Hence no additional incomplete end occurs, and \(\widetilde{\cX^o}\) is complete for the Hodge metric. Since \(\bB^n\) is complete, connected and simply connected, the standard complete local-isometry argument gives
\[
\widetilde{\cX^o}\simeq\bB^n.
\]
The \(\rho\)-equivariance of \(\mathcal{P}\) identifies the orbifold deck group with the discrete subgroup
\[
\Gamma=\rho\bigl(\pi_1^{\rm orb}(\cX^o)\bigr)\subset \PU(n,1),
\]
so \(\rho\) is faithful and \(\Gamma\) is discrete. The volume of \(\bB^n/\Gamma\) equals the volume of \(\cX^o\) for the Hodge metric. Lemma~\ref{lem:cusp-complete-volume} gives finite volume near \(D^p\), and compactness away from \(D^p\) gives finite total volume. Hence \(\Gamma\) is a lattice.

Proposition~\ref{prop:standard-toroidal-identification} identifies \((\cX,D^p)\) with the canonical orbifold toroidal compactification of \([\bB^n/\Gamma]\). This proves the theorem.
\end{proof}

\subsection{Local monodromy}

\begin{proposition}\label{prop:peripheral}
Peripheral loops around \(D^p\) map to parabolic elements of \(\PU(n,1)\). Local orbifold loops around \(D_i^c\) map to finite-order elliptic elements; the normal character has order \(p_i\).
\end{proposition}

\begin{proof}
Along \(D^p\), the harmonic bundle is tame nilpotent with trivial parabolic structure. Hence the local monodromy around a component of \(D^p\) is unipotent. The residue is non-zero because the Higgs residue sends the logarithmic frame \(d\log z_1\) to \(1\). Therefore the corresponding element of \(\PU(n,1)\) is a non-trivial unipotent isometry of complex hyperbolic space, hence parabolic.

Along \(D_i^c\), the divisor is not removed. It is part of the orbifold structure. On the root chart \(z_i=w_i^{p_i}\), the local stabilizer \(\mu_{p_i}\) acts by \(w_i\mapsto\zeta w_i\). The period normal coordinate is \(\xi_i=w_i^{p_i-q_i}\), so the induced action on the target normal coordinate is
\[
\xi_i\longmapsto \zeta^{p_i-q_i}\xi_i.
\]
Because \(\gcd(p_i,p_i-q_i)=\gcd(p_i,q_i)=1\), this character has order \(p_i\). Hence the local monodromy is elliptic of order \(p_i\) in the normal direction.
\end{proof}

\begin{proof}[Proof of Proposition~\ref{prop:branched-equality}]
Proposition~\ref{prop:flat-adjoint} gives flatness of the adjoint harmonic bundle. Proposition~\ref{prop:principal-system} supplies the principal \(\PU(n,1)\)-system, and Proposition~\ref{prop:period} gives the period map. Proposition~\ref{prop:branching} gives the ramification index, while Proposition~\ref{prop:E10-vs-cotangent} gives the unramified criterion. Proposition~\ref{prop:cover-manifold} proves that the orbifold universal cover is represented by a simply connected complex manifold and that compact local stabilizers act faithfully in the normal period coordinates. The assertions on \(D^p\) and on peripheral monodromy follow from Propositions~\ref{prop:Dp-smooth} and~\ref{prop:peripheral}.
\end{proof}

\section{The Converse}\label{sec:converse}

The converse is formulated on the coarse space, with the branching data measured on the root charts.

Put
\[
   X^\circ:=X\setminus(D^p\cup D^c),\qquad
   \Delta:=D^p+\sum_{i\in I}\frac{q_i}{p_i}D_i^c.
\]
The coefficients of \(D^p\) are equal to one, while the compact coefficients are the parabolic weights.

\begin{definition}\label{def:mixed-growth}
Let \(T\) be a closed positive \((1,1)\)-current on \(X\) such that \(T|_{X^\circ}\) is a smooth K\"ahler form. We say that \(T\) has at most mixed Poincar\'e--cone growth of type \(\Delta\) if, near every point of \(D^p+D^c\), there are coordinates
\[
   (z_1,\ldots,z_r,x_1,\ldots,x_s,y_1,\ldots,y_t)
\]
with \(D^p=\{z_1\cdots z_r=0\}\) and \(D^c=\{x_1\cdots x_s=0\}\), such that, as currents,
\[
   0\leq T\leq C\,\omega_{\rm mix},
\]
where
\[
   \omega_{\rm mix}:=
   \sum_{\alpha=1}^r \frac{\sqrt{-1}\,dz_\alpha\wedge d\bar z_\alpha}
        {|z_\alpha|^2(\log |z_\alpha|^2)^2}
   +\sum_{\beta=1}^s \frac{\sqrt{-1}\,dx_\beta\wedge d\bar x_\beta}{|x_\beta|^{2a_\beta}}
   +\sum_\gamma \sqrt{-1}\,dy_\gamma\wedge d\bar y_\gamma,
\]
and, if \(\{x_\beta=0\}=D_{i_\beta}^c\), then \(a_\beta=q_{i_\beta}/p_{i_\beta}\).
\end{definition}

\begin{definition}\label{def:admissible}
A class \(\alpha\in H^{1,1}(X,\mathbb R)\) is called \(\Delta\)-admissible if \(\alpha\) is big and nef and contains a closed positive current \(T\in\alpha\) satisfying the two conditions of Definition~\ref{def:mixed-growth}.
\end{definition}

\begin{lemma}\label{lem:mixed-current-big-nef}
Let \(T\) be a closed positive \((1,1)\)-current on \(X\) such that \(T|_{X^\circ}\) is a smooth K\"ahler form and \(T\) has at most mixed Poincar\'e--cone growth of type \(\Delta\). Then the cohomology class \([T]\) is big and nef. Consequently \([T]\) is \(\Delta\)-admissible.
\end{lemma}

\begin{proof}
We first record the zero-Lelong-number estimate. Lelong numbers are monotone for positive currents, so it suffices to check the local model terms in Definition~\ref{def:mixed-growth}. The Poincar\'e term has potential \(-\log(-\log |z|^2)\), whose quotient by \(\log |z|^2\) tends to zero. For the cone term, write \(a=1-\beta\), \(\beta>0\); up to a positive constant,
\[
   \sqrt{-1}\,\partial\bar\partial |x|^{2\beta}
      = \frac{\sqrt{-1}\,dx\wedge d\bar x}{|x|^{2a}}
\]
on the punctured disc, and the potential \(|x|^{2\beta}\) is continuous at the origin. Hence all model terms, and therefore \(T\), have zero Lelong number along \(D^p+D^c\). Since \(T\) is smooth on \(X^\circ\), it has zero Lelong number there as well. Thus \(T\) has zero Lelong number at every point of \(X\), and no component of \(D^p+D^c\) occurs in the Siu divisorial part of \(T\).

Fix a smooth K\"ahler form \(\omega_0\) on \(X\), and choose a smooth representative \(\theta\) of the class \([T]\). Write
\[
   T=\theta+\sqrt{-1}\,\partial\bar\partial\varphi
\]
for a quasi-plurisubharmonic potential \(\varphi\).

By Demailly's regularization theorem with control of Lelong numbers \cite{Demailly1992}, for every \(\varepsilon>0\) there exists a closed current \(T_\varepsilon\in[T]\) with analytic singularities such that
\[
   T_\varepsilon\geq -\varepsilon\omega_0,
\]
and whose Lelong numbers are bounded by those of \(T\). Since all Lelong numbers of \(T\) vanish, the analytic singularities of \(T_\varepsilon\) are removable. Thus \(T_\varepsilon\) is represented by a smooth real \((1,1)\)-form in \([T]\) satisfying the same lower bound. Hence \([T]\) is nef.

It remains to prove bigness. Choose a coordinate ball \(B\Subset X^\circ\). Since \(T|_{X^\circ}\) is K\"ahler, after shrinking \(B\) there is a constant \(c>0\) such that
\[
   T\geq c\omega_0
\]
on \(B\). The regularization may be chosen to converge to \(T\) smoothly on compact subsets of the locus where \(T\) is smooth; hence, for all sufficiently small \(\varepsilon\),
\[
   T_\varepsilon+\varepsilon\omega_0\geq \frac{c}{2}\omega_0
\]
on \(B\), while \(T_\varepsilon+\varepsilon\omega_0\geq0\) on all of \(X\). Therefore
\[
   ([T]+\varepsilon[\omega_0])^n
   =\int_X (T_\varepsilon+\varepsilon\omega_0)^n
   \geq \int_B (T_\varepsilon+\varepsilon\omega_0)^n
   \geq \left(\frac{c}{2}\right)^n\int_B\omega_0^n>0.
\]
Letting \(\varepsilon\to0\) gives \([T]^n>0\). Since \([T]\) is nef and \(X\) is projective, the standard nef volume criterion implies that \([T]\) is big; see, for instance, \cite{DemaillyPaun2004}. The same current \(T\) satisfies Definition~\ref{def:mixed-growth}, so \([T]\) is \(\Delta\)-admissible by Definition~\ref{def:admissible}.
\end{proof}

\begin{definition}\label{def:branched-structure}
We say that the weighted pair \((X,D^p,D^c,\{q_i/p_i\}_{i\in I})\) carries a branched complex-hyperbolic structure of type \(\Delta\) if the following hold.
\begin{enumerate}[label=\textup{(\alph*)}]
\item On the root stack \(\cX=X[\sqrt[p_i]{D_i^c}]_{i\in I}\), obtained by taking the \(p_i\)-th root along every \(D_i^c\), the orbifold universal cover
\[
   \widetilde{\cX^o}\longrightarrow \cX^o
\]
is represented by a simply connected smooth complex manifold.
\item There are a representation \(\rho:\pi_1^{\rm orb}(\cX^o)\to \PU(n,1)\) and a \(\rho\)-equivariant holomorphic period map
\[
   \mathcal{P}:\widetilde{\cX^o}\longrightarrow \mathbb B^n
\]
whose differential is non-degenerate away from the inverse image of \(D^c\). Near a lift of \(D_i^c\), in the root coordinate \(x_i=w_i^{p_i}\) and a normal target coordinate \(\xi_i\), one has
\[
   \xi_i\circ\mathcal{P}=u_i(w)w_i^{m_i},\qquad u_i(0)\neq0,
   \qquad m_i=p_i-q_i.
\]
\item Along \(D^p\), the descended hyperbolic metric has at most Poincar\'e cusp growth.
\end{enumerate}
\end{definition}

\begin{proof}[Proof of Proposition~\ref{prop:converse}]
\noindent\emph{Proof of \textup{(1).}}
The Bergman form is \(\PU(n,1)\)-invariant, so
\(\mathcal{P}^*\omega_{\mathbb B}\) is invariant under the orbifold deck
group and descends to a smooth K\"ahler form on \(X^\circ\). Near \(D_i^c\),
use \(x_i=w_i^{p_i}\) and the local normal form
\(\xi_i=u_i(w)w_i^{m_i}\). In the normal direction,
\[
   d\xi_i=(\text{unit})\,w_i^{m_i-1}dw_i.
\]
Since the Bergman metric is smooth and positive in the target coordinate, the pull-back normal part is mutually bounded by
\[
   |w_i|^{2m_i-2}\sqrt{-1}\,dw_i\wedge d\bar w_i.
\]
Using \(m_i=p_i-q_i\) and \(x_i=w_i^{p_i}\), this is, up to a positive smooth factor,
\[
   |x_i|^{-2q_i/p_i}\sqrt{-1}\,dx_i\wedge d\bar x_i.
\]
The cusp hypothesis gives the Poincar\'e bound along \(D^p\).
This proves \textup{(1)}.

\medskip
\noindent\emph{Proof of \textup{(2).}}
Because the mixed model is locally integrable, \(\omega_{\rm br}\) has locally
finite mass near \(|D|\). The form is closed and positive on \(X^\circ\); by
the Skoda--El Mir extension theorem for closed positive currents
\cite{DemaillyBook} its trivial extension across \(|D|\) is a closed positive
current. We denote the extension again by \(\omega_{\rm br}\), and put
\[
   \Theta:=\frac{n+1}{2\pi}\,\omega_{\rm br}.
\]

It remains to identify the de Rham class of \(\Theta\). Put
\[
   a_i:=\frac{q_i}{p_i},\qquad
   L:=K_X+D^p+\sum_{i\in I} a_i\,D_i^c.
\]
On \(X^\circ\), the complex-hyperbolic metric satisfies
\[
   \Ric(\omega_{\rm br})=-(n+1)\omega_{\rm br}.
\]
Let \(h_K\) be the Hermitian metric on \(K_X|_{X^\circ}\) induced by the volume
form \(\omega_{\rm br}^n\). The standard curvature formula gives
\[
   c_1(K_X,h_K)
   =
   -\frac{1}{2\pi}\Ric(\omega_{\rm br})
   =
   \Theta
\]
on \(X^\circ\).

Choose an integer \(N>0\) such that \(Na_i\in\mathbb Z\) for all \(i\), and set
\[
   L_N:=N L.
\]
Then \(L_N\) is a holomorphic line bundle. It suffices to prove
\([N\Theta]=c_1(L_N)\). Write
\[
   D^p=\sum_{\ell\in J_p}D_\ell^p,\qquad D^c=\sum_{i\in I}D_i^c.
\]
Let \(s_\ell^p\) and \(s_i^c\) be the canonical sections of
\(\mathcal O_X(D_\ell^p)\) and \(\mathcal O_X(D_i^c)\), respectively, and
choose arbitrary smooth Hermitian metrics \(h_\ell^p\) and \(h_i^c\) on these
divisor line bundles. On \(X^\circ\) define a singular Hermitian metric on
\(L_N\) by
\[
   H_N
   :=
   h_K^{\otimes N}
   \otimes
   \bigotimes_{\ell\in J_p}
   \left(\frac{h_\ell^p}{|s_\ell^p|_{h_\ell^p}^2}\right)^{\otimes N}
   \otimes
   \bigotimes_{i\in I}
   \left(\frac{h_i^c}{|s_i^c|_{h_i^c}^2}\right)^{\otimes Na_i}.
\]
On \(X^\circ\) these
divisor factors have zero curvature, since their canonical sections have
constant norm \(1\). Hence
\[
   c_1(L_N,H_N)=N\,c_1(K_X,h_K)=N\Theta
\]
on \(X^\circ\).

Now we check the singularities of \(H_N\). Choose local coordinates
\[
   D^p=\{z_1\cdots z_r=0\},\qquad
   D^c=\{x_1\cdots x_s=0\}.
\]
In local frames with \(s_\alpha^p=z_\alpha e_\alpha^p\) and
\(s_\beta^c=x_\beta e_\beta^c\), the divisor factors contribute
\(|z_\alpha|^{-2N}\) and \(|x_\beta|^{-2Na_\beta}\). The mixed
Poincar\'e--cone model gives the divisorial part of the determinant metric in
the form
\[
   |\sigma_K|_{h_K}^2
   =
   \prod_{\alpha=1}^r |z_\alpha|^2
   \prod_{\beta=1}^s |x_\beta|^{2a_\beta}
   \cdot \Lambda,
\]
where \(\sigma_K\) is a local frame of \(K_X\), and \(\log\Lambda\) has only
log-log growth in the cusp variables and bounded, equivalently
orbifold-bounded, growth in the compact root directions. Therefore, in the
corresponding local frame of \(L_N\),
\[
   |\sigma_{L_N}|_{H_N}^2=\Lambda^N.
\]
Thus \(H_N\) has locally integrable weights, with no remaining
\(\log|z_\alpha|^2\) or \(\log|x_\beta|^2\) divisorial term. Its curvature
current is well-defined and has no extra Siu divisorial mass along
\(D^p+D^c\). The trivial extension of \(\Theta\) also has no divisorial mass
by the mixed growth estimate and the zero-Lelong-number estimate in Lemma~\ref{lem:mixed-current-big-nef}. Since
the two currents agree on \(X^\circ\), they agree on all of \(X\):
\[
   c_1(L_N,H_N)=N\Theta.
\]

Finally, let \(H_0\) be any smooth Hermitian metric on \(L_N\). Since
\(H_N=H_0e^{-\varphi}\) for a global \(L^1_{\rm loc}\) function \(\varphi\),
\[
   c_1(L_N,H_N)
   =
   c_1(L_N,H_0)
   +
   \frac{\sqrt{-1}}{2\pi}\partial\bar\partial\varphi.
\]
The second term is exact as a current. Hence
\[
   [N\Theta]=[c_1(L_N,H_N)]=c_1(L_N),
\]
and division by \(N\) gives
\[
   [\Theta]=c_1(K_X+\Delta)
   =c_1\left(K_X+D^p+\sum_{i\in I}\frac{q_i}{p_i}D_i^c\right).
\]

We have also proved that \(\Theta\) has at most mixed Poincar\'e--cone growth
and is a smooth K\"ahler form on \(X^\circ\). Lemma~\ref{lem:mixed-current-big-nef}
then shows that \(c_1(K_X+\Delta)\) is big and nef, and that it is
\(\Delta\)-admissible.
This proves \textup{(2)}.

\medskip
\noindent\emph{Proof of \textup{(3).}}
Equip \(E|_{X^\circ}=\Omega_X^1|_{X^\circ}\oplus\mathcal O_X\) with the Hodge metric
\[
   h=\omega_{\rm br}^{-1}\oplus h_0,
\]
where \(h_0\) is the constant metric on \(\mathcal O_X\). This is the pull-back of the standard homogeneous Hodge metric on the system of Hodge bundles over \(\mathbb B^n\). Hence the projective Hitchin--Simpson connection is flat. Equivalently, if
\[
   F_{h,\theta}:=F_h+[\theta,\theta_h^\dagger]
\]
denotes the Hitchin--Simpson curvature of \((E,\theta,h)\), then
\begin{equation}\label{eq:hs-flat-conv}
   F_{h,\theta}^{\perp}=0.
\end{equation}
Here \(\perp\) denotes the trace-free part of the Hitchin--Simpson curvature. The local model gives
\[
   |dx_i|_h^2\sim |x_i|^{2q_i/p_i}
\]
in the conormal direction to \(D_i^c\), exactly the growth prescribed by the parabolic weight \(q_i/p_i\). Along \(D^p\) the metric is the usual tame nilpotent cusp model. Thus \(h\) is adapted to \(E_*\) and acceptable with respect to the mixed Poincar\'e--cone model.

Let \(\alpha\) be a \(\Delta\)-admissible big and nef class, and choose a current \(T\in\alpha\) as in Definition~\ref{def:admissible}. Put \(\omega_T:=T|_{X^\circ}\), which is a smooth K\"ahler form on \(X^\circ\). Let \(S_*\subset E_*\) be a saturated Higgs subsheaf of rank \(s\), and put \(r=\rk E=n+1\). Let
\[
   Z\subset X
\]
be the union of \(|D|\) and the locus where \(E/S\) is not locally free. Then \(Z\setminus |D|\) has codimension at least two, and on
\[
   U:=X\setminus Z
\]
the sheaf \(S\) is a holomorphic subbundle of \(E\) preserved by \(\theta\). Denote by \(\Pi\) the \(h\)-orthogonal projection from \(E|_U\) to \(S|_U\), and by \(h_S\) the induced metric on \(S|_U\).

We use the following Chern--Weil formula for parabolic degrees.

Let
\[
   b_j(S):=\sum_a a\,\rk\bigl(\Gr^j_a(S_*)\bigr),
\]
where \(\Gr^j_a(S_*)\) is the graded piece of the parabolic filtration of \(S_*\)
along \(D_j^c\). Thus
\[
   \parc_1(S_*)=c_1(S)+\sum_j b_j(S)[D_j^c].
\]
Choose a finite Galois cover \(g:Y\to X\) adapted to the denominators of the
weights, so that \(g^*D_j^c=m_j\,p_j\widetilde D_j^c\). By \cite{Biswas1997}, the pull-back of every
locally abelian parabolic bundle is an equivariant orbifold bundle. Let
\(\widetilde S\) be the equivariant bundle corresponding to \(S_*\), and let
\(\widetilde h_S\) be the metric induced by \(h_S\) in the root frames. Then
\[
   c_1(\widetilde S)=g^*\parc_1(S_*),
\]
and, on \(Y\setminus g^{-1}(D^p+D^c)\),
\[
   \tr F_{\widetilde h_S}=g^*\tr F_{h_S}.
\]
Locally this is the change from the coarse adapted frame to the root
frame. If \(x_j=w_j^{p_j}\) and the weight is \(a=q_j/p_j\), the factor
\(|x_j|^{2a}\) in the coarse adapted metric becomes a smooth factor in the
orbifold frame \(w_j^{-q_j}dx_j\). Hence the divisor contribution in
\(\parc_1(S_*)\) is exactly converted into the ordinary first Chern class of
\(\widetilde S\) on the cover.

Put \(\widetilde T:=g^*T\). Since \(T\) is dominated by the mixed
Poincar\'e--cone model, \(\widetilde T\) has at most Poincar\'e growth along
\(g^{-1}D^p\) and locally finite mass along the inverse
image of \(D^c\). A cut-off argument gives the required intersection formula.

\begin{lemma}\label{lem:par-cw-admissible}
With the notation above, put \(\widetilde Z:=g^{-1}(Z)\). Then
\[
   c_1(\widetilde S)\cdot [\widetilde T]^{n-1}
   =
   \frac{\sqrt{-1}}{2\pi}
   \int_{Y\setminus \widetilde Z}
   \tr F_{\widetilde h_S}\wedge \widetilde T^{n-1}.
\]
The same formula holds for \(\widetilde E\) with the pulled-back metric
\(\widetilde h\), where \(\widetilde E\) is the equivariant bundle associated
with \(E_*\).
\end{lemma}

\begin{proof}
Put
\[
   B_Y:=g^{-1}(D^p+D^c),\qquad
   A:=\widetilde Z\setminus B_Y,\qquad
   V:=Y\setminus\widetilde Z.
\]
The set \(A\) has codimension at least two. On \(V\) all bundles and metrics
are smooth, so the issue is to identify the improper Chern--Weil integral with
the intersection number.

Near a compact component of \(D^c\), write the local adapted cover in the
normal direction as
\[
   x=w^\ell,
\]
where \(\ell\) is divisible by the denominator \(p\) of the weight \(a=q/p\).
The cone part of the mixed model gives
\[
   0\le T\le C\frac{\sqrt{-1}\,dx\wedge d\bar x}{|x|^{2a}}+\cdots.
\]
Therefore
\[
   0\le \widetilde T
      \le C |w|^{2\ell(1-a)-2}
         \sqrt{-1}\,dw\wedge d\bar w+\cdots
      =
      C |w|^{2(\ell/p)(p-q)-2}
        \sqrt{-1}\,dw\wedge d\bar w+\cdots.
\]
Since \(\ell/p\ge1\) and \(p-q\ge1\), the exponent is non-negative.

Near a cusp component of \(D^p\), a local normal cover has the form
\[
   z=t^\ell.
\]
The Poincar\'e factor is stable under such a finite cover:
\[
   g^*\left(
   \frac{\sqrt{-1}\,dz\wedge d\bar z}
        {|z|^2(\log |z|^2)^2}
   \right)
   =
   \frac{\sqrt{-1}\,dt\wedge d\bar t}
        {|t|^2(\log |t|^2)^2}.
\]
Thus, in local coordinates \((t_\alpha,w_\beta,y_\gamma)\) on \(Y\), with
\(t_\alpha=0\) over \(D^p\) and \(w_\beta=0\) over \(D^c\), there is a model
form
\[
   \Omega:=
   \sum_\alpha
   \frac{\sqrt{-1}\,dt_\alpha\wedge d\bar t_\alpha}
        {|t_\alpha|^2(\log |t_\alpha|^2)^2}
   +
   \sum_\beta \sqrt{-1}\,dw_\beta\wedge d\bar w_\beta
   +
   \sum_\gamma \sqrt{-1}\,dy_\gamma\wedge d\bar y_\gamma
\]
such that \(0\le \widetilde T\le C\Omega\). In particular, on
\(Y\setminus B_Y\) the smooth forms \(\widetilde T^k\), \(1\le k\le n\), have
locally finite mass near \(B_Y\), and their trivial extensions give no mass to
\(B_Y\cup A\). The domination by \(\Omega\) also implies that pairing these
extensions with smooth closed forms computes the cohomological intersections
with \([\widetilde T]^k\).

The induced metric \(\widetilde h_S\) has acceptable growth. In an adapted
holomorphic frame for \(\widetilde E\), the compact weights have
been absorbed by the cover, while the cusp directions have only logarithmic
tame growth. Hence, if \(H_S\) is the matrix of \(\widetilde h_S\) on a local
frame of \(\widetilde S\), then for some \(C,N>0\)
\[
   C^{-1}\prod_\alpha s_\alpha^{-N}I
   \le H_S\le
   C\prod_\alpha s_\alpha^N I,
   \qquad s_\alpha=-\log |t_\alpha|^2,
\]
away from \(A\). Since \(S\) is saturated, the possible degeneration of
\(\det H_S\) away from the boundary is supported on \(A\). If \(r_A\) denotes
a local distance to \(A\), then
\[
   \partial\log\det H_S
   =
   O\left(
      \sum_\alpha \frac{dt_\alpha}{t_\alpha s_\alpha}
      +\frac{dr_A}{r_A}
      +\sum_\beta dw_\beta+\sum_\gamma dy_\gamma
   \right).
\]
Consequently
\[
   \tr F_{\widetilde h_S}\wedge\widetilde T^{n-1}
\]
has finite mass on \(V\).

Let \(k_S\) be a smooth Hermitian metric on \(\widetilde S\), and put
\[
   \gamma_S:=\frac{\sqrt{-1}}{2\pi}\tr F_{\widetilde h_S},
   \qquad
   \gamma_0:=\frac{\sqrt{-1}}{2\pi}\tr F_{k_S}.
\]
On \(V\),
\[
   \gamma_S-\gamma_0
   =
   dd^c\varphi,
   \qquad
   \varphi:=\log\frac{\det\widetilde h_S}{\det k_S},
   \qquad
   dd^c=\frac{\sqrt{-1}}{2\pi}\partial\bar\partial.
\]
The estimates above show that \(\varphi\in L^1_{\rm loc}\), has no divisorial
logarithmic term along \(B_Y\), and has only logarithmic singularities along
the codimension at least two set \(A\). Thus the trivial extension of
\(\gamma_S\) is a closed \(L^1\)-current representing
\[
   [\gamma_S]=[\gamma_0]=c_1(\widetilde S).
\]

Choose cut-off functions \(\chi_\varepsilon\) which are equal to zero near
\(B_Y\cup A\), equal to one away from a slightly larger neighbourhood, and
satisfy the usual logarithmic estimates. In a cusp variable one may arrange
\[
   |\partial\chi_\varepsilon|
   \le
   \frac{C}{|t|\,|\log |t||}
\]
on its support. In a compact root variable the ordinary logarithmic cut-off is
enough because \(\widetilde T\) is locally bounded there and
\(\partial\varphi\) has no \(dw/w\)-type divisorial pole. Near \(A\), the
standard radial logarithmic cut-off and the codimension at least two estimate
for \(dr_A/r_A\) give the same conclusion. Therefore
\[
   \lim_{\varepsilon\to0}
   \int_Y
   d\chi_\varepsilon\wedge d^c\varphi\wedge\widetilde T^{n-1}=0.
\]
Since all forms are smooth on \(\operatorname{supp}\chi_\varepsilon\subset V\)
and \(d\widetilde T=0\) there, Stokes' theorem gives
\[
   \int_Y
   \chi_\varepsilon\,dd^c\varphi\wedge\widetilde T^{n-1}
   =
   -\int_Y
   d\chi_\varepsilon\wedge d^c\varphi\wedge\widetilde T^{n-1}.
\]
Letting \(\varepsilon\to0\), the exact Bott--Chern term pairs trivially with
\(\widetilde T^{n-1}\). Therefore
\[
\begin{aligned}
   \int_V\gamma_S\wedge\widetilde T^{n-1}
   &=
   \int_Y\gamma_0\wedge\widetilde T^{n-1}  \\
   &=
   c_1(\widetilde S)\cdot[\widetilde T]^{n-1}.
\end{aligned}
\]
The argument for \(\widetilde E\) is identical, without a codimension-two
non-locally-free locus.
\end{proof}

Dividing by the degree of the cover and using the change-of-variables formula
for the finite map \(g\), one obtains
\begin{equation}\label{eq:par-degree-cw-current}
   \pardeg_\alpha(S_*)
   =
   \frac{\sqrt{-1}}{2\pi}
   \int_U \tr F_{h_S}\wedge T^{n-1},
   \qquad
   \pardeg_\alpha(E_*)
   =
   \frac{\sqrt{-1}}{2\pi}
   \int_U \tr F_h\wedge T^{n-1}.
\end{equation}

Because \(S\) is \(\theta\)-invariant, we have
\[
   \tr[\theta_S,\theta_{S,h_S}^\dagger]=0.
\]

Thus, in the trace appearing in \eqref{eq:par-degree-cw-current}, \(F_{h_S}\) may
equivalently be replaced by the Hitchin--Simpson curvature
\[
   F_{h_S,\theta_S}=F_{h_S}+[\theta_S,\theta_{S,h_S}^\dagger].
\]

On \(U\), the standard second fundamental form identity for Higgs bundles gives
\begin{equation}\label{eq:second-fund-current}
\begin{aligned}
   \frac{\sqrt{-1}}{2\pi}\tr F_{h_S,\theta_S}
   &-\frac{s}{r}\frac{\sqrt{-1}}{2\pi}\tr F_{h,\theta}  \\
   &=\frac{\sqrt{-1}}{2\pi}\tr\bigl(\Pi F_{h,\theta}^{\perp}\bigr)
     -\mathcal Q(\Pi).
\end{aligned}
\end{equation}
Here \(\mathcal Q(\Pi)\) is the non-negative \((1,1)\)-form whose contraction with
\(\omega_T:=T|_{X^\circ}\) is
\[
   \Lambda_{\omega_T}\mathcal Q(\Pi)
   = c_n\Bigl(|\bar\partial_E\Pi|^2_{h,\omega_T}+|[\theta,\Pi]|^2_{h,\omega_T}\Bigr)
\]
for a positive dimensional constant \(c_n\). This formula is the usual Chern--Weil formula
for the induced Higgs metric: the first norm is the ordinary second fundamental form, and
the second norm measures the failure of the orthogonal splitting to be Higgs-invariant.

After integrating \eqref{eq:second-fund-current} against \(T^{n-1}\) on \(U\),
the parabolic Chern--Weil formula
\eqref{eq:par-degree-cw-current} identifies the left-hand side with
\(\pardeg_\alpha(S_*)-\frac{s}{r}\pardeg_\alpha(E_*)\). Lemma~\ref{lem:par-cw-admissible} accounts for the boundary and codimension-two contributions. Using \eqref{eq:hs-flat-conv}, we obtain
\begin{equation}\label{eq:slope-current-identity}
\begin{aligned}
   \pardeg_\alpha(S_*)-\frac{s}{r}\pardeg_\alpha(E_*)
   &=-C_n\int_U
      \Bigl(|\bar\partial_E\Pi|^2_{h,\omega_T}+|[\theta,\Pi]|^2_{h,\omega_T}\Bigr)
      \frac{\omega_T^n}{n!}  \\
   &\le 0,
\end{aligned}
\end{equation}
where \(C_n>0\). Dividing by the ranks gives
\[
   \mu_\alpha(S_*)\le \mu_\alpha(E_*).
\]
Thus \((E_*,\theta)\) is \(\mu_\alpha\)-semistable.

If equality holds in \eqref{eq:slope-current-identity}, both non-negative terms in the integral vanish. Since \(\omega_T\) is a K\"ahler metric on \(X^\circ\), one obtains
\[
   \bar\partial_E\Pi=0,
   \qquad
   [\theta,\Pi]=0
\]
on \(U\). Therefore \(\Pi\) is a holomorphic Higgs projection on \(U\). Since \(S\) and \(E/S\) are saturated, the two summands extend uniquely across the codimension at least two set by reflexivity. The local growth of \(h\) along \(D^p+D^c\) shows that the extended summands are compatible with the induced parabolic filtrations; equivalently, the splitting is a splitting in the category of parabolic Higgs sheaves. Hence any saturated Higgs subsheaf with the same \(\alpha\)-slope is a direct summand.
This proves \textup{(3)}.

\medskip
\noindent\emph{Proof of \textup{(4).}}
The parabolic Chern classes are computed by the same adapted Chern--Weil interpretation as in Lemma~\ref{lem:par-cw-admissible}, applied on an adapted cover and then divided by the degree of the cover. On \(X^\circ\), equation~\eqref{eq:hs-flat-conv} says that the Hitchin--Simpson curvature of \((E,\theta,h)\) is scalar. Therefore the induced Hitchin--Simpson curvature on \(\End_0(E)\), and hence also on \(\End(E)\), is zero. The adapted-cover Chern--Weil representative of \(\parch_2(\End E_*)\) is consequently zero on the complement of the boundary; the growth estimates used above show that no boundary current is produced. Thus
\[
   \parch_2(\End E_*)=0.
\]
For a rank \(r=n+1\) bundle,
\[
   \parch_2(\End E_*)=2r\,\parch_2(E_*)-\parch_1(E_*)^2,
\]
which gives the asserted equality.
\end{proof}

\begin{proof}[Proof of Theorem~\ref{thm:standard-orbifold-converse}]
The orbifold ball quotient structure gives the orbifold universal cover
\[
   \bB^n\longrightarrow \cX^o
\]
and a period map which is the identity map on \(\bB^n\), up to the action of \(\Gamma\). Near a compact orbifold divisor of stabilizer order \(p_i\), the standard orbifold local model has root coordinate \(z_i=w_i^{p_i}\), and the normal ball coordinate is a unit multiple of \(w_i\). Since \(q_i=p_i-1\), this is exactly the local normal form
\[
   \xi_i=u_i(w)w_i^{p_i-q_i}=u_i(w)w_i,
   \qquad u_i(0)\ne0.
\]
At the cusp boundary \(D^p\), the toroidal compactification gives the usual Poincar\'e cusp growth. Hence the weighted pair \((X,D^p,D^c,\{(p_i-1)/p_i\})\) satisfies Definition~\ref{def:branched-structure} with
\[
   \Delta=D^p+\sum_{i\in I}\left(1-\frac1{p_i}\right)D_i^c.
\]
Proposition~\ref{prop:converse} therefore gives that \(K_X+\Delta\) is big and nef, that \(c_1(K_X+\Delta)\) is \(\Delta\)-admissible, and that the asserted polystability and Chern equality hold for every \(\Delta\)-admissible big and nef class.
\end{proof}

\section{Functoriality and the Categorical Equivalence}\label{sec:categorical}

In this section we prove the functoriality needed for the categorical equivalence stated in Theorem~\ref{thm:categorical}. We first record the local consequences of the definition of morphisms in \(\BG_n^{\rm std}\).

\begin{proposition}\label{prop:BG-morphism-local-form}
Let
\[
   f:\cX_1^{\log}\longrightarrow \cX_2^{\log}
\]
be a morphism in \(\BG_n^{\rm std}\). Then the following hold.
\begin{enumerate}[label=\textup{(\roman*)}]
\item The induced map on the open orbifold parts
\[
   f^o:\cX_1^o\longrightarrow \cX_2^o
\]
is finite \'etale. If the open orbifolds are connected and \(k=\deg(f^o)\), then \(f_*\pi_1^{\rm orb}(\cX_1^o)\) is a subgroup of \(\pi_1^{\rm orb}(\cX_2^o)\) of index \(k\).

\item At the generic point of a component of \(D_1^p\) mapping to a component of \(D_2^p\), one can choose local cusp coordinates \(q_1,q_2\) such that
\[
   f^*q_2=u q_1^m,\qquad u(0)\ne0,
\]
where \(m\ge1\). Thus the ordinary map is ramified to order \(m\) in the normal cusp direction, while the logarithmic differential satisfies
\[
   f^*d\log q_2=m\,d\log q_1+d\log u.
\]

\item At the generic point of a compact root divisor \(D_{1,j}^c\) mapping to \(D_{2,i}^c\), let the root orders be \(p_{1,j}\) and \(p_{2,i}\). Then \(p_{1,j}\mid p_{2,i}\), the stabilizer map
\[
   \mu_{p_{1,j}}\hookrightarrow\mu_{p_{2,i}}
\]
is injective, and in root coordinates the morphism is unramified in the normal direction. Equivalently, after choosing normal root coordinates \(w_1,w_2\),
\[
   f^*w_2=u w_1,\qquad u(0)\ne0.
\]
On coarse coordinates \(z_{1,j}=w_1^{p_{1,j}}\) and \(z_{2,i}=w_2^{p_{2,i}}\), this becomes
\[
   f^*z_{2,i}=u' z_{1,j}^{p_{2,i}/p_{1,j}},
   \qquad u'(0)\ne0.
\]

\item Near a point lying over \(D^p\cap D^c\), the cusp logarithmic directions and the compact root directions have product log-root form after choosing suitable coordinates; in particular no cusp factor occurs in a compact root coordinate and no compact root factor occurs in a cusp coordinate.
\end{enumerate}
\end{proposition}

\begin{proof}
On \(\cX^o\) the logarithmic boundary \(D^p\) has been removed, so \(df_{\log}\) is the ordinary differential on root-stack charts. Since this differential is an isomorphism by Definition~\ref{def:BGcategory}, the finite representable morphism \(f^o\) is unramified on orbifold charts, hence finite \'etale. The index statement is the standard covering-space statement for a connected finite \'etale orbifold covering of degree \(k\).

For the cusp assertion, let \(q_2=0\) be a local equation of the target cusp component. The support condition
\[
   \bigl(f^{-1}(|D_2^p|)\bigr)_{\rm red}=|D_1^p|
\]
implies, at the generic point of a component of \(D_1^p\) over it, that
\[
   f^*q_2=u q_1^m
\]
for a unit \(u\) and an integer \(m\ge1\). Applying \(d\log\) gives the displayed formula. Thus the possible ordinary ramification along \(D^p\) is absorbed by the logarithmic cotangent direction. The integer \(m\) is a local Kummer ramification index and is not, in general, the degree \(k\) of \(f^o\).

For compact root divisors, work on root charts. Let \(w_2=0\) be a target root divisor. Since \(D^c\) is not logarithmic, \(dw_2\) is an ordinary cotangent direction inside \(\Omega^1_{\cX_2}(\log D_2^p)\). The isomorphism \(df_{\log}\) forces the pull-back of this normal direction to have non-zero linear part in the corresponding source root normal direction. Hence, after choosing a source root coordinate \(w_1\), one has \(f^*w_2=u w_1\) with \(u\) a unit. Representability gives an injective homomorphism of stabilizers \(\mu_{p_{1,j}}\hookrightarrow\mu_{p_{2,i}}\), so \(p_{1,j}\mid p_{2,i}\). Passing to the coarse coordinates \(z_{1,j}=w_1^{p_{1,j}}\) and \(z_{2,i}=w_2^{p_{2,i}}\) gives the stated formula.

Finally, the support condition forbids compact root factors in the pull-back of a cusp coordinate. If the pull-back of a compact root coordinate had a cusp factor, its ordinary differential would vanish along the cusp divisor in the compact root direction, contradicting the invertibility of \(df_{\log}\). After a change of local coordinates, the cusp and compact root directions therefore have product log-root form.
\end{proof}

\begin{proposition}\label{prop:monodromy-functoriality}
Let
\[
   f:\cX_1^{\log}\longrightarrow \cX_2^{\log}
\]
be a morphism in \(\BG_n^{\rm std}\), and let
\[
   \rho_a:\pi_1^{\rm orb}(\cX_a^o)\longrightarrow \PU(n,1),\qquad a=1,2,
\]
be the monodromy representations obtained from the equality-case construction. Then, after conjugating \(\rho_1\) in \(\PU(n,1)\),
\[
   \rho_1=\rho_2\circ f_*.
\]
In particular, if \(\Gamma_a=\rho_a(\pi_1^{\rm orb}(\cX_a^o))\), then \(f\) determines a finite-index inclusion
\[
   \Gamma_1\hookrightarrow \Gamma_2.
\]
If \(f^o\) is \(k\)-sheeted, then this inclusion has index \(k\).
\end{proposition}

\begin{proof}
By Proposition~\ref{prop:BG-morphism-local-form}, the map
\[
   f^o:\cX_1^o\longrightarrow \cX_2^o
\]
is a finite \'etale orbifold covering. If it is \(k\)-sheeted, then it induces an injection
\[
   f_*:\pi_1^{\rm orb}(\cX_1^o)\hookrightarrow \pi_1^{\rm orb}(\cX_2^o).
\]
Its image has index \(k\).

For \(a=1,2\), Theorem~\ref{thm:standard-uniformization} identifies the orbifold universal cover \(\widetilde{\cX_a^o}\) with \(\bB^n\). Choose period biholomorphisms
\[
   \mathcal P_a:\widetilde{\cX_a^o}\longrightarrow \bB^n
\]
which are \(\rho_a\)-equivariant. A lift of \(f^o\) to universal covers
\[
   \widetilde f:\widetilde{\cX_1^o}\longrightarrow \widetilde{\cX_2^o}
\]
is a biholomorphism, since it is a covering map between simply connected orbifold universal covers. Therefore
\[
   A:=\mathcal P_2\circ\widetilde f\circ\mathcal P_1^{-1}
\]
belongs to \(\Aut(\bB^n)=\PU(n,1)\).

For every \(\gamma\in\pi_1^{\rm orb}(\cX_1^o)\), the lift satisfies
\[
   \widetilde f\circ\gamma=f_*(\gamma)\circ\widetilde f.
\]
Using the equivariance of the two period maps gives
\[
   A\rho_1(\gamma)A^{-1}=\rho_2(f_*(\gamma)).
\]
After conjugating \(\rho_1\) by \(A\), this is precisely
\[
   \rho_1=\rho_2\circ f_*
\]
and the finite-index inclusion \(\Gamma_1\hookrightarrow\Gamma_2\) follows. Since the representations \(\rho_a\) are faithful by Theorem~\ref{thm:standard-uniformization}, the index of the lattice inclusion is the index of \(f_*\pi_1^{\rm orb}(\cX_1^o)\) in \(\pi_1^{\rm orb}(\cX_2^o)\), namely \(k\).
\end{proof}

\begin{proposition}\label{prop:lattice-inclusion-boundary}
Let \(\Gamma_1\subset \Gamma_2\) be a finite-index inclusion representing a morphism in \(\Lat_n^{\rm reg}\). Put
\[
   U_a:=\bB^n/\Gamma_a,\qquad a=1,2,
\]
and let
\[
   \mathcal T_{\Gamma_a}
   :=
   \overline U_{\Gamma_a}^{\rm tor}
   \Bigl[\sqrt[p_{a,i}]{D_{a,i}^c}\Bigr]_{i\in I_a}
\]
be its canonical root stack, where \(D_{a,i}^c\) are the compact elliptic divisors with root orders \(p_{a,i}\). Then the finite covering
\[
   U_1\longrightarrow U_2
\]
extends uniquely to a finite representable morphism
\[
   (\mathcal T_{\Gamma_1},D_{\Gamma_1}^{\rm tor})
   \longrightarrow
   (\mathcal T_{\Gamma_2},D_{\Gamma_2}^{\rm tor})
\]
which is a local isomorphism in the log-root sense. In particular, it is a morphism in \(\BG_n^{\rm std}\).
\end{proposition}

\begin{proof}
Choose a neat finite-index normal subgroup
\[
   \Lambda\triangleleft\Gamma_2,\qquad \Lambda\subset\Gamma_1.
\]
This is obtained by taking a sufficiently small neat subgroup of \(\Gamma_1\) and then passing to its core in \(\Gamma_2\). Put
\[
   F_a:=\Gamma_a/\Lambda,\qquad a=1,2,
\]
and let \((\overline U_\Lambda^{\rm tor},D_\Lambda^{\rm tor})\) be the smooth toroidal compactification of \(\bB^n/\Lambda\). Then \(F_1\subset F_2\), and the coarse toroidal compactifications are
\[
   \overline U_{\Gamma_a}^{\rm tor}=\overline U_\Lambda^{\rm tor}/F_a.
\]
Thus the open covering \(U_1\to U_2\) extends to the finite quotient morphism
\[
   \overline U_\Lambda^{\rm tor}/F_1
   \longrightarrow
   \overline U_\Lambda^{\rm tor}/F_2.
\]
This extension is unique, since the compactifications are normal and separated and the map is already fixed on the dense open ball quotients.

We next check the boundary form. Let \(\xi\in\partial\bB^n\) be a cusp representative. For the neat group \(\Lambda\), a toroidal neighbourhood of the corresponding boundary component is the total space of a line bundle
\[
   L_\Lambda\longrightarrow A_\Lambda
\]
over an abelian variety, with boundary the zero section \(A_\Lambda\). In a local trivialization write \(t\) for the normal coordinate. For \(a=1,2\), the cusp stabilizer
\[
   \Gamma_{a,\xi}:=\Gamma_a\cap\operatorname{Stab}(\xi)
\]
contains \(\Lambda_\xi\) with finite index. The central lattice in the unipotent radical determines the normal toroidal coordinate. Since \(\Gamma_{1,\xi}\subset\Gamma_{2,\xi}\), the corresponding central lattice for \(\Gamma_1\) is a finite-index subgroup of that for \(\Gamma_2\). If this local index is \(m\), then, after choosing compatible coordinates on the two quotients,
\[
   q_2=u q_1^m,\qquad u(0)\ne0.
\]
The unit \(u\) depends on the trivialization of the normal line bundle. It does not affect the logarithmic cotangent direction, since
\[
   d\log q_2=m\,d\log q_1+d\log u
\]
and \(d\log u\) is regular. The tangential part is a finite \'etale morphism of the abelian boundary quotients. Hence the morphism is Kummer log-\'etale along the cusp boundary, and
\[
   \bigl(f^{-1}(|D_{\Gamma_2}^{\rm tor}|)\bigr)_{\rm red}
   =
   |D_{\Gamma_1}^{\rm tor}|.
\]
The integer \(m\) is a local central-lattice index; if several source cusp components lie over the same target component, the global degree is the sum of the corresponding normal indices times the tangential boundary degrees.

Now consider a compact elliptic divisor. By the regular log-root condition, on the common neat toroidal cover \(\overline U_\Lambda^{\rm tor}\) one may choose a normal coordinate \(w\) to the mirror such that the stabilizer in \(F_a\) is a cyclic reflection group of order \(p_a\),
\[
   w\longmapsto \zeta w,\qquad \zeta\in\mu_{p_a}.
\]
Since \(F_1\subset F_2\), the source stabilizer injects into the target stabilizer, so \(p_1\mid p_2\). If \(z_a\) is the corresponding coarse coordinate on \(\overline U_\Lambda^{\rm tor}/F_a\), then
\[
   z_a=w^{p_a}.
\]
Therefore the quotient morphism has local coarse form
\[
   z_2=w^{p_2}=(w^{p_1})^{p_2/p_1}=z_1^{p_2/p_1},
\]
up to multiplication by a unit. On the root stacks, both sides use \(w\) as the normal root coordinate, so the morphism is \'etale in root coordinates and representable by the injectivity of the stabilizer map.

Finally, at points lying over the intersection of the toroidal boundary and compact mirrors, the regular log-root condition gives product coordinates on \(\overline U_\Lambda^{\rm tor}\),
\[
   (t,w_1,\ldots,w_s,y),
\]
where \(t=0\) is the toroidal boundary and \(w_i=0\) are the compact mirrors. On the two quotients the local coordinates have the form
\[
   q_a=t^{m_a},\qquad z_{a,i}=w_i^{p_{a,i}},
\]
together with \'etale tangential coordinates. Hence the quotient map has product local form
\[
   q_2=u q_1^m,\qquad
   z_{2,i}=v_i z_{1,i}^{p_{2,i}/p_{1,i}},\qquad
   y_2=\varphi(y_1),
\]
where \(u\) and \(v_i\) are units and \(\varphi\) is \'etale. Thus cusp logarithmic directions and compact root directions do not mix. The displayed local forms show exactly that the resulting morphism is a local isomorphism in the log-root sense.
\end{proof}

\begin{proposition}\label{prop:quotient-functor}
The compactified quotient construction defines a functor
\[
   Q_{\rm tor}:\Lat_n^{\rm reg}\longrightarrow \BG_n^{\rm std}.
\]
\end{proposition}

\begin{proof}
Let \(\Gamma\) be an object of \(\Lat_n^{\rm reg}\). By Definition~\ref{def:canonical-orbifold-toroidal} and the regular log-root condition, the quotient \([\bB^n/\Gamma]\) has a canonical orbifold toroidal compactification
\[
   (\mathcal T_\Gamma,D_\Gamma^{\rm tor}).
\]
Its coarse space is smooth, the cusp boundary is ordinary logarithmic, and the compact elliptic divisors carry standard root orders \(p_i\). Theorem~\ref{thm:standard-orbifold-converse} applies to this standard unramified orbifold complex-hyperbolic structure and shows that the canonical parabolic Higgs bundle is polystable and satisfies the parabolic Bogomolov--Gieseker equality. Thus the compactification is an object of \(\BG_n^{\rm std}\).

A morphism \(\Gamma_1\to\Gamma_2\) is represented, after conjugation, by a finite-index inclusion \(\Gamma_1\subset\Gamma_2\). The quotient map
\[
   \bB^n/\Gamma_1\longrightarrow \bB^n/\Gamma_2
\]
is finite \'etale on the open orbifold quotients. By Proposition~\ref{prop:lattice-inclusion-boundary}, it extends to a finite representable morphism of the canonical compactifications which is a local isomorphism in the log-root sense. Hence it gives a morphism in \(\BG_n^{\rm std}\), and the construction is compatible with composition.
\end{proof}

\begin{proposition}\label{prop:monodromy-functor}
The monodromy construction defines a functor
\[
   U_{\rm BG}:\BG_n^{\rm std}\longrightarrow \Lat_n^{\rm reg},\qquad
   \cX^{\log}\longmapsto \rho_{\cX}\bigl(\pi_1^{\rm orb}(\cX^o)\bigr).
\]
\end{proposition}

\begin{proof}
For an object \(\cX^{\log}\in\BG_n^{\rm std}\), Theorem~\ref{thm:standard-uniformization} produces a finite-volume lattice
\[
   \Gamma_\cX=\rho_\cX\bigl(\pi_1^{\rm orb}(\cX^o)\bigr)\subset\PU(n,1)
\]
and identifies \((\cX,D^p)\) with the canonical orbifold toroidal compactification of \([\bB^n/\Gamma_\cX]\). Equivalently, \(\Gamma_\cX\) is the orbifold deck group transported to \(\Aut(\bB^n)=\PU(n,1)\) by a period biholomorphism; changing the biholomorphism conjugates \(\Gamma_\cX\). The local root-stack structure of \(\cX\), together with the ordinary logarithmic structure along \(D^p\), shows that \(\Gamma_\cX\) has regular log-root type. Proposition~\ref{prop:monodromy-functoriality} gives the action on morphisms, and compatibility with composition follows from functoriality of the induced maps on orbifold fundamental groups.
\end{proof}

\begin{proof}[Proof of Theorem~\ref{thm:categorical}]
The functors \(Q_{\rm tor}\) and \(U_{\rm BG}\) are defined in Propositions~\ref{prop:quotient-functor} and~\ref{prop:monodromy-functor}. If \(\Gamma\in\Lat_n^{\rm reg}\), then \(Q_{\rm tor}(\Gamma)\) has open orbifold part \([\bB^n/\Gamma]\), and the monodromy representation obtained from the canonical parabolic Higgs bundle is the standard inclusion of \(\Gamma\) into \(\PU(n,1)\), up to conjugation. Hence
\[
   U_{\rm BG}(Q_{\rm tor}(\Gamma))\simeq \Gamma.
\]

Conversely, if \(\cX^{\log}\in\BG_n^{\rm std}\), Theorem~\ref{thm:standard-uniformization} identifies \((\cX,D^p)\) with the canonical orbifold toroidal compactification of \([\bB^n/\Gamma_\cX]\). Therefore
\[
   Q_{\rm tor}(U_{\rm BG}(\cX^{\log}))\simeq \cX^{\log}.
\]

It remains to identify morphisms. A finite-index inclusion \(\Gamma_1\subset g\Gamma_2g^{-1}\) gives a finite orbifold covering of the open ball quotients. After conjugating, Proposition~\ref{prop:lattice-inclusion-boundary} shows that this covering extends uniquely to a finite representable morphism of the canonical compactifications which is a local isomorphism in the log-root sense. This is the morphism \(Q_{\rm tor}\) assigns to the inclusion. Conversely, any morphism in \(\BG_n^{\rm std}\) gives, by Proposition~\ref{prop:monodromy-functoriality}, a finite-index inclusion of the associated monodromy lattices. These two constructions are inverse to one another, because both are determined by the induced map on the common universal cover \(\bB^n\). Thus the sets of morphisms are identified, and the two functors are quasi-inverse equivalences of categories.
\end{proof}

\section*{Acknowledgements}
The authors thank Shiyu Zhang for helpful discussions. They also thank Xiaojin Lin and Mao Sheng for reading the manuscript and for suggestions on the categorical equivalence.

\printbibliography

\end{document}